\documentclass[conference,letterpaper]{IEEEtran}
\IEEEoverridecommandlockouts
\usepackage{cite}
\usepackage{amsmath,amssymb,amsfonts}
\usepackage{algorithm}
\usepackage{algorithmic}
\usepackage{graphicx}
\usepackage{textcomp}
\usepackage{xcolor}
\usepackage{amssymb, graphicx, amsmath, amsthm}
\usepackage{tikz}
\usepackage{multicol, blindtext}
\usepackage{subfigure}
\newtheorem{theorem}{Theorem}
\newtheorem{lem}{Lemma}
\newtheorem{cor}{Corollary}

\newtheorem{rem}{Remark}
\newtheorem{question}{Question}

\def\BibTeX{{\rm B\kern-.05em{\sc i\kern-.025em b}\kern-.08em
    T\kern-.1667em\lower.7ex\hbox{E}\kern-.125emX}}
\begin{document}
\title{The Distribution of the Number of Real Solutions to the Power Flow Equations\\

\thanks{The authors gratefully acknowledge support from the National Science Foundation under grant DMS 1735928.}
}

\author{\IEEEauthorblockN{Julia Lindberg\IEEEauthorrefmark{1},
Alisha Zachariah\IEEEauthorrefmark{2},
Nigel Boston\IEEEauthorrefmark{1}\IEEEauthorrefmark{2}, and
Bernard Lesieutre\IEEEauthorrefmark{1}}
\IEEEauthorblockA{\IEEEauthorrefmark{1}Department of Electrical and Computer Engineering,
University of Wisconsin-Madison, Madison, WI 53706 USA}
\IEEEauthorblockA{\IEEEauthorrefmark{2}Department of Mathematics,
University of Wisconsin-Madison, Madison, WI 53706 USA}
\thanks{Corresponding author: J. Lindberg, jrlindberg@wisc.edu.}}

\markboth{Journal of \LaTeX\ Class Files,~Vol.~14, No.~8, August~2015}%
{Shell \MakeLowercase{\textit{et al.}}: Bare Demo of IEEEtran.cls for IEEE Transactions on Magnetics Journals}
\maketitle

\begin{abstract}
 In this paper we study the distributions of the number of real solutions to the power flow equations over varying electrical parameters. We introduce a new monodromy and parameter homotopy continuation method for quickly finding all solutions to the power flow equations. We apply this method to find distributions of the number of real solutions to the power flow equations and compare these distributions to those of random polynomials. It is observed that while the power flow equations tend to admit many fewer real-valued solutions than a bound on the total number of complex solutions, for low levels of load they tend to admit many more than a corresponding random polynomial. We show that for cycle graphs the number of real solutions can achieve the maximum bound for specific parameter values and for complete graphs with four or more vertices there are susceptance values that give infinitely many real solutions.  
\end{abstract}

\section{ Introduction}

The power flow equations, a system of quadratic equations relating node voltages to active and reactive power injections at each node, are ubiquitous in all studies of electric power networks. Solutions to the power flow equations provide operating points for power networks, which are important for informing decisions from future planning concerns involving capital investments, to day-to-day resource scheduling and market operations, and to real-time stability analyses. While the study of these equations is rich, much is still unknown about the solutions to these equations.

The theoretical study of the power flow solutions can be broadly classified as work bounding the number of complex solutions \cite{baillieul1982geometric,chen2018on,mehta2016numerical,chen2018counting}, methods for finding all of the solutions \cite{molzahn2013counterexample,tavora1972stability,ma1993an,mehta2016numerical,lesieutre2015an,salam1989parallel,wu_2017} and more recently studying the distributions of the number of real solutions \cite{zachariah2018distributions,lindberg2018the, lesieutre2019on}. 
Practical studies have focused on finding a small number of particular types of solutions \cite{Liu2005,Lee2009,Klump2000}.

This paper primarily aims to compute distributions of the number of real solutions to the power flow equations by varying network parameters.  Explicitly calculating all real solutions tends to be computationally challenging -- methods to do so do not scale well with network size. For instance, the traditional approach uses a homotopy method that traces from known solutions for an easy polynomial system to those of the more complex power system model. This method first finds all of the complex solutions then selects those that are real. These methods generally are not even efficient in finding all complex solutions in the sense that they use an easy polynomial system with many more solutions than the desired power flow model. 

In this paper we introduce a new monodromy and parameter homotopy continuation method to find all complex solutions to the power flow equations. This technique allows us to exploit the symmetry present in the equations to improve computational speed. We then use this algorithm to find distributions of the number of real solutions to the power flow equations for cyclic graphs up to ten vertices and complete graphs up to eight vertices.

This paper is organized as follows: In Section \ref{sec2} we describe the power flow equations and the modeling assumptions we make. In Section \ref{sec3} we outline standard homotopy methods used to solve polynomial systems as well as introduce our monodromy and parameter continuation algorithm. We then compare the running time of each method. In Section \ref{sec4} we discuss the distribution of the number of real solutions to the power flow equations for varying susceptance values. Motivated by a result in applied algebraic geometry \cite{rouillier1999solving}, we compare these distributions to that of the real roots of a corresponding random polynomial and find that the power flow equations admit many more real solutions than a random polynomial. We also show that for cyclic networks the maximum number of real valued solutions, namely the complex bound given in \cite{chen2018counting}, can always be attained. We conclude in Section \ref{sec5} with families of networks that give special solution sets.

\section{Model and Approach} \label{sec2}

\subsection{The Power Flow Equations}\label{sec:IIA}
\indent We model an $n$-node electric power network as a connected, undirected graph, $G = (V, E)$, where each vertex $v_m \in V$, $0 \leq m \leq n-1$, represents a node (bus) in the power network. There is an edge, $e_{km}$ between vertices $v_k$ and $v_m$ if the corresponding nodes in the power network are connected. Each edge has a known complex admittance $b_{km} + j g_{km}$ where $b_{km}, g_{km} \in \mathbb{R}$ and $j$ denotes $\sqrt{-1}$. Each vertex $v_k$ has an associated complex power injection $P_k + j Q_k$, $P_k, Q_k \in \mathbb{R}$ where $P_k$ models the active power and $Q_k$ models the reactive power.

At each node, $k$, the relationship between the active and reactive power flows is captured by the nonlinear relations
\begin{align}
    P_k &= \sum_{m=0}^{n-1} V_k V_m (g_{km} \cos(\theta_k - \theta_m) + b_{km} \sin (\theta_k - \theta_m) ) \label{trigpfeqs1} \\
    Q_k &= \sum_{m=0}^{n-1} V_k V_m (g_{km} \sin (\theta_k - \theta_m) + b_{km} \cos(\theta_k - \theta_m)) \label{trigpfeqs2}
\end{align}
where $V_k$ is the voltage magnitude and $\theta_k$ represents the voltage angle at node $k$. We fix $v_0$ to be the \textit{slack bus}, designating $\theta_0 = 0$. We take the admittance $b_{km} + j g_{km}$ to be zero if the vertices $k$ and $m$ are not connected. Equations \eqref{trigpfeqs1}-\eqref{trigpfeqs2} are the \textit{power flow equations}, and for fixed admittances, there are four quantities associated with each node: voltage magnitude,$V_k$, voltage angle, $\theta_k$, active power injection $P_k$, and reactive power injection $Q_k$. 

As typical of traditional power flow studies, the equations for each node are solved for two of the nodal quantities when the other two are specified.
In this paper we consider a power network where all nodes have unknown reactive power injections but maintain constant voltage magnitude and active power is given: $Q_k$ and $\theta_k$ are unknown while $P_k$ and $|V_k|$, are known constants. The single exception is the slack bus for which $V_0$ and $\theta_0=0$ are specified. 

We can make this system purely algebraic by introducing the change of variables $x_k = V_k \cos (\theta_k)$ and $y_k = V_k \sin (\theta_k)$. Under this transformation $(\ref{trigpfeqs1})$ becomes the system of $2(n-1)$ equations in $2(n-1)$ variables
\begin{align}
P_k &= \sum_{m=0}^{n-1}  g_{km}(x_kx_m+y_ky_m)+b_{km}(x_ky_m-x_my_k) \label{fullpfeq1} \\
V_k^2 &= x_k^2+y_k^2 \label{fullpfeq2}
\end{align}
Equations \eqref{fullpfeq1} and \eqref{fullpfeq2} are the power flow equations for the special case where all nodes are PV nodes. The \textit{power  flow  problem} is to compute all of the real-valued solutions to the system of equations given by varying $k$. At the slack node $(x_0,y_0) = (V_0,0)$.

\subsection{Assumption: Lossless and Zero Power Injections}
In practice $g_{km} \ll b_{km}$ so, for the rest of this paper, we assume that the power network is $\mathit{lossless}$, meaning $g_{km}=0$ for all branches. In addition, we will assume 
zero power injections, i.e. $P_k = 0$ for all nodes $v_k$. The assumption of zero power injections is made because this condition typically admits the most real-valued solutions \footnote{Certain rare examples have been found where this fails to be true \cite{zachariah2018distributions}}. As we examine distributions with respect to electrical parameters, this low (zero) power injection model offers the greatest number of solutions to represent in this manner. 
Without loss of generality we assume $V_k = 1$ for each $k$.  With this normalization, the system of equations under consideration is 
\begin{align}
    x_k^2 + y_k^2 &= 1 \label{pfeq1}\\
    \sum_{m=0}^{n-1} b_{km}(x_ky_m - x_m y_k) &= 0 \label{pfeq2}
\end{align}
for $k = 1,\ldots,n-1$ where $x_0 = 1$ and $y_0 = 0$. Therefore, the only parameters of the system are the susceptances, $b_{km}$.

Under the above assumptions, any $n$-node system has $2^{n-1}$ ``trivial" solutions corresponding to setting $y_k=0$ and  having $x_k = \pm 1$. So, for any set of susceptances these power flow equations always admit $2^{n-1}$ real-valued, trivial solutions. 
We would like to determine the distribution of the number of nontrivial real solutions to \eqref{pfeq1}-\eqref{pfeq2} for susceptance values chosen at \textit{random}.

\subsection{Distribution of the Number of Real Solutions}

We now define what it means to pick susceptance values at random. Since the susceptances are real valued, the space of susceptances is not compact. It may then seem that there is no natural choice for a distribution (or, equivalently, a finite measure) on this space. However, because \eqref{pfeq1}-\eqref{pfeq2} are homogeneous, the number of real solutions is unaffected by simultaneously scaling all susceptance values. Therefore, without loss of generality we can normalize the susceptance values to lie on a unit sphere, which is compact. It is then natural to pick susceptances uniformly at random from the unit sphere. 

Let $N$ denote the number of real solutions for a random choice of susceptances. For any non-negative integer value of $N$, the set of tuples of susceptances that yield that many real solutions is a closed subset of the sphere and its measure is thereby well-defined. In other words, $N$ is a random variable in its own right, and we are interested in studying its distribution.

\section{A New Method for Counting Real Solutions} \label{sec3}
We would like to repeatedly solve \eqref{pfeq1}-\eqref{pfeq2} for randomly chosen susceptance values to obtain an empirical distribution on the number of nontrivial real solutions. Numerical methods for solving polynomial systems have been around for decades. We briefly outline the main idea below but give \cite{sturmfels2002solving,sommese2005the} as more detailed references.

Consider a system of polynomial equations
\begin{align*}
    F(x) &= \{ p_1(x_1,\ldots, x_m), p_2(x_1,\ldots,x_m) ,\ldots, p_m(x_1,\ldots, x_m) \} = 0.
\end{align*}{}
for which we assume that the number of solutions to $F(x) = 0$ is finite. 
The main idea is to construct a \textit{homotopy}
\begin{align}
    H(x;t) &= \gamma (1-t)G(x) + t F(x)
\end{align}
such that: 
\begin{enumerate}
    \item The solutions to $G(x)=0$ are trivial to find, 
    \item There are no singularities along the path $t \in [0,1)$, and
    \item All isolated solutions of $F(x) = 0$ can be reached \cite{li1997numerical}.
    \end{enumerate}
By using a random $\gamma \in \mathbb{C}$ each path for $t \in [0,1)$ avoids singularities almost surely, so condition $2$ is easily met. This is referred to as the \textit{gamma trick} \cite{mehta2016numerical}. Continuation methods are then used to track the solutions from $G(x) = 0$ to $F(x)=0$ as $t$ varies from $0$ to $1$.  \textit{predictor-corrector methods} are commonly used to compute the paths \cite{butcher2003numerical}.

The system $G(x)=0$ is called the \textit{start system} and there are many choices for it. A \textit{total degree} start system is 
\begin{align}
    G(x) &= \{x_1^{d_1}-1,\ldots, x_n^{d_n}-1 \} =0
\end{align}{}
where $d_i$ is the degree of $p_i$. The number of solutions to $G(x)=0$ is $d_1\cdots d_n$, which is the Bezout bound. This means that you have to track $d_1\cdots d_n$ paths in order to get all solutions to $F(x) = 0$. If $F(x)=0$ has close to $d_1\cdots d_n$ solutions this is a reasonable start system.

If $F(x)$ is sparse, the number of solutions to $F(x) = 0$ can be much less than $d_1\cdots d_n$ so tracking $d_1\cdots d_n$ paths is wasteful computation. In this case it is often more computationally efficient to use a \textit{polyhedral} start system. These homotopy algorithms rely on the Bernstein-Kushnirenko-Khovanskii (BKK) bound \cite{bernshtein1979the, kouchnirenko1976polyedres, khovanskii1978newton}, which gives an upper bound on the number of isolated $\mathbb{C}^* = \mathbb{C} \backslash \{0\}$ solutions for polynomial systems. This upper bound is called the \textit{mixed volume} of the system. For sparse polynomial systems the mixed volume can be much smaller than the Bezout bound. Huber and Sturmfels proposed the first polyhedral homotopy algorithm that achieves this bound by deforming the start system to a system with number of solutions equal to the mixed volume of the original system \cite{huber1995a}. The main disadvantage to polyhedral homotopy methods is that the start systems may not be as easy to solve as in the total degree case. There is still the potential for wasted computation here as the mixed volume of a system might not be a tight upper bound on the number of solutions.

In the case of the power flow equations the number of $\mathbb{C}^*$ solutions typically grows exponentially as the size of the network increases, so no matter the start system used, the number of paths that need to be tracked will also grow exponentially. Standard homotopy methods, including those described above, don't consider the symmetry in the equations or the fact that we can ignore the trivial solutions. We highlight a new method below that seeks to reduce this computational burden.

\subsection{Monodromy and Parameter Homotopy Continuation} \label{sec31}
Aiming to reduce the number of paths tracked, we introduce a monodromy and parameter homotopy algorithm that exploits the symmetry in the power flow equations. It is outlined in Algorithm 1 but explained in more detail below.
\begin{algorithm}
\begin{itemize}
\item \textbf{Input}: An undirected graph $G = (V,E)$, one choice of susceptances $b$ 
\item \textbf{Output}: All $\mathbb{C}^*$ solutions to the power flow equations with susceptances $b$ 
\item \textbf{Preprocessing Step}: 
\begin{enumerate}
\item Find one solution to the power flow equations for one choice of susceptance values $\hat{b} \in \mathbb{C}^{|E|}$
\begin{itemize}
    \item For $k = 1,\ldots, n-1$ pick random $x_k \in \mathbb{C}$ and set $y_k = \sqrt{1 - x_k^2}$ so $(x_k,y_k)$ satisfy \eqref{pfeq1} $\forall$ $k$.
    \item Plugging these choices of $x_k,y_k$ for $k = 1\ldots, n-1$ into the system of equations defined by \eqref{pfeq2}  gives an underdetermined linear system of equations in the susceptances.  Find one solution $\hat{b}$ to this system.
    \item Output one solution $(x,y)$ to the power flow equations for susceptances $\hat{b}$
\end{itemize}
\item Use Monodromy to find remaining nontrivial solutions to the power flow equations for susceptances $\hat{b}$ up to the equivalence $(x_1,\ldots,x_{n-1},y_1,\ldots,y_{n-1} ) \sim (x_1,\ldots,x_{n-1},-y_1,\ldots,-y_{n-1}) $. Call this solution set $S_{\hat{b}}$.
\end{enumerate}
\item[] \textbf{Procedure:}
\begin{enumerate}
\item Use Parameter Homotopy to track $S_{\hat{b}}$ from susceptances $\hat{b} \in \mathbb{C}^{|E|}$ to desired solution set $S_b$ for $b\in \mathbb{R}^{|E|}$
\end{enumerate}
\end{itemize}
\caption{}
\end{algorithm}

Monodromy methods work by taking one solution to a system of polynomial equations and finding all solutions. We omit the details here and give \cite{duff2019solving,campo2017critical,amendola2016solving} as a general reference and \cite{lindberg2020exploiting} as a reference specific to the power flow equations. 

In order for monodromy methods to find all solutions for a given set of susceptances, the variety defined by the power flow equations for that set of susceptances must be \textit{irreducible}, meaning it cannot decompose into a union of algebraic sub-varieties.

\begin{lem} \label{irreduciblelemma}
 The nontrivial solutions to $(\ref{pfeq1})-(\ref{pfeq2})$ form an irreducible variety.
\end{lem}{}
\begin{proof}
See Lemma 3.1 in \cite{lindberg2020exploiting}.
\end{proof}{}

\begin{cor}
    If \eqref{pfeq1}-\eqref{pfeq2} has nonzero active power injections, then the variety corresponding to these equations is irreducible.
\end{cor}
\begin{proof}
As in Lemma 3.1 of \cite{lindberg2020exploiting} we consider the change of coordinates $x_i = \frac{2t_i}{1+t_i^2}$ and $y_i = \frac{1-t_i^2}{1+t_i^2}$. This transforms \eqref{pfeq1}-\eqref{pfeq2} into for $k=1,\ldots,n-1$
\begin{align}
P_k &= \sum_{m=0}^{n-1} b_{km} \Big( \frac{2t_k(1-t_m^2)-2t_m(1-t_k^2)}{(1+t_k^2)(1+t_m^2)} \Big) \label{powerinjectionparameqs}
\end{align}
for $P_k \neq 0$. As explained in the proof of Lemma 3.1 in \cite{lindberg2020exploiting}, it suffices to show that for almost all $t = (t_1,\ldots, t_{n-1}) \in \mathbb{C}^{n-1}$ there exists a $P_1,\ldots,P_{n-1} \in \mathbb{R}$ and $b \in \mathbb{R}^{|E|}$ that is a solution to \eqref{powerinjectionparameqs}. This system of equations is linear in the susceptances $b_{km}$ and active power injections $P_k$ for $m,k=1,\ldots,n-1$ so we can write it as $Ab = P$ where $P = (P_1,\ldots,P_{n-1})$ and $A\in \mathbb{C}^{n-1 \times |E|}$ is a weighted incidence matrix of $G$. This matrix generically has rank $n-1$, meaning for almost all choices of $t \in \mathbb{C}^{n-1}$, $Ab = P$ has a solution.
\end{proof}

This means that using monodromy methods it is theoretically possible to find all solutions to the power flow equations given just one solution. One downside of monodromy is that unless the number of $\mathbb{C}^*$ solutions is known, there is no stopping criterion for this algorithm to terminate. We denote the number of $\mathbb{C}^*$ solutions as $K$. For a fixed network, $K$ is generically independent of the susceptance values. For complete graphs $K_n$ and cyclic graphs $C_n$ on $n$ vertices with real-valued power injections, \cite{baillieul1982geometric,chen2018counting} prove that the number of $C^*$ solutions is $\binom{2n-2}{n-1}$ and $n\binom{n-1}{\lfloor \frac{n-1}{2}\rfloor}$ respectively\footnote{We note that while these bounds were proven for networks with nonzero active power injections, they are still valid under our assumption of zero active power injections. This is because zeroing out the constant terms does not change the Jacobian of the system and therefore won't force the system onto the discriminant locus.}. This provides an upper bound for $K$ and can be used as a stopping criterion for the monodromy method. For other graphs where the number of $\mathbb{C}^*$ solutions is not known, a common practice is to terminate the calculation after there have been $10$ loops without finding any new solutions. At this point you can run a \textit{trace test} to verify that there are no other solutions \cite{hauenstein2020multiprojective}. For a fixed network, we think of this monodromy step as a preprocessing step in the case of the power flow equations. 

Once we find all solutions, $S_{\hat{b}}$, to the power flow equations for one choice of susceptances $\hat{b}$, we use parameter continuation methods to track these solutions to solutions $S_{b}$ for our desired choice of susceptances, $b$. Parameter continuation works as follows. Consider a system of parametric polynomial equations
\begin{align*}
    F(x,\hat{b}) = \{f_1(x,\hat{b}),\ldots,f_N(x,\hat{b}) \} = 0
\end{align*}{}
where $\hat{b}\in \mathbb{C}^m$ are the parameters and $x \in \mathbb{C}^N$ are the variables. In the case of the power flow equations on a graph $G = (V, E)$ with $|V| = n$ and $|E| = m$, we have $\hat{b} \in \mathbb{C}^m$ and $x \in \mathbb{C}^{2n-2}$. We then construct a homotopy to our target parameters $b \in \mathbb{R}^m$
\begin{align*}
    H(x;t) = F(x, \frac{\gamma_1(1-t)\hat{b} + \gamma_2 tb}{t \gamma_2 + (1-t)\gamma_1})
\end{align*}{}
where $t$ runs from $0$ to $1$. Again, we choose random $\gamma_1,\gamma_2 \in \mathbb{C}$ to avoid singularities. This is an efficient homotopy method in the sense that for every solution in $S_{b}$ we track exactly one path from $S_{\hat{b}}$. 

So far we have explained a basic monodromy preprocessing step paired with a parameter continuation algorithm that will do at least as well, in terms of number of paths tracked, as polyhedral homotopy methods. For the case of counting the number of real solutions to the power flow equations, we have additional information that we would like to exploit. We first would like to avoid any computation needed to find the trivial solutions. Since these solutions lie in their own subvariety, the monodromy method outlined above will not find these solutions and the parameter continuation algorithm will not track these paths. In addition we observe the following symmetry in solutions.
\begin{lem}
 If $(x_1,\ldots,x_{n-1},y_1,\ldots,y_{n-1})$ is a solution to \eqref{pfeq1}-\eqref{pfeq2}, so is $(x_1,\ldots,x_{n-1},-y_1,\ldots,-y_{n-1})$.
\end{lem}{}
\begin{proof}
See Lemma $3.2$ in \cite{lindberg2020exploiting}.
\end{proof}{}
This symmetry allows us to speed up computation even more. Instead of tracking all paths in $S_{\hat{b}}$, we only need to track half of them.  This leads to a major reduction in the number of paths we need to track. For bipartite networks, we observe even more symmetry.

\begin{lem}
\label{symmetrybipartitelemma}
Let $G = (V,E)$ be a bipartite graph with disjoint vertex sets $S,T \subset V$ that partition $V$ where for all $e = v_mv_k \in E$, $v_m \in S$ and $v_k \in T$ or vice versa. Without loss of generality, say $v_0,\ldots,v_s \in S$ and $v_{s+1},\ldots,v_{n-1} \in T$. Then if $(x_1,\ldots,x_{n-1},y_1,\ldots,y_{n-1})$ is a solution to \eqref{pfeq1}-\eqref{pfeq2} so is
\begin{enumerate}
    \item $(x_1,\ldots,x_{n-1},-y_1,\ldots,-y_{n-1})$
    \item $(-x_1,\ldots,-x_s,x_{s+1},\ldots,x_{n-1},\\
    y_1,\ldots,y_s,-y_{s+1},\ldots,-y_{n-1})$
    \item $(-x_1,\ldots,-x_s, x_{s+1},\ldots,x_{n-1},\\
    -y_1,\ldots,-y_s,y_{s+1},\ldots,y_{n-1})$
\end{enumerate}
\end{lem}
\begin{proof}
See Lemma $3.3$ in \cite{lindberg2020exploiting}.
\end{proof}

In the case of nonzero active power injections, bipartite graphs still have some symmetry.

\begin{cor}
\label{symmetrybipartitelemma}
Let $G = (V,E)$ be a bipartite graph with disjoint vertex sets $S,T \subset V$ that partition $V$ where for all $e = v_mv_k \in E$, $v_m \in S$ and $v_k \in T$ or vice versa. Without loss of generality, say $v_0,\ldots,v_s \in S$ and $v_{s+1},\ldots,v_{n-1} \in T$. Consider \eqref{pfeq1}-\eqref{pfeq2} with nonzero active power injections:
\begin{align}
P_k &= \sum_{m=0}^{n-1} b_{km}(x_ky_m - x_m y_k)  \label{nonzeropowerinjections}
\end{align}
$P_k \neq 0$, for $k = 1,\ldots,n$. Then if $(x_1,\ldots,x_{n-1},y_1,\ldots,y_{n-1})$ is a solution to \ref{pfeq2} so is $(-x_1,\ldots,-x_s,x_{s+1},\ldots,x_{n-1},y_1,\ldots,y_s,-y_{s+1},\ldots,-y_{n-1})$.
\end{cor}
\begin{proof}
Substituting in $(-x_1,\ldots,-x_s,x_{s+1},\ldots,x_{n-1},\\y_1,\ldots,y_s,-y_{s+1},\ldots,-y_{n-1})$ to \eqref{nonzeropowerinjections}, the result is clear.
\end{proof}

The comparison between the number of paths needed to track using our new modified parameter, polyhedral and total degree homotopy methods is shown in Figure $\ref{fig:paths}$. We compare running times in Section \ref{sec33}.

\begin{figure}[h!]
    \centering
    \includegraphics[width =0.5\textwidth]{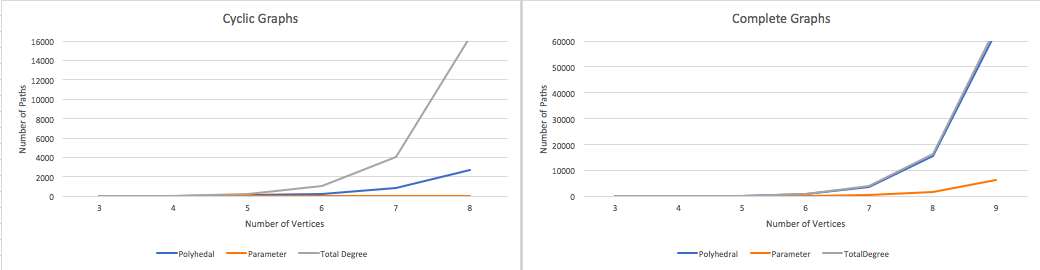}
    \caption{Number of Paths Needed to track for Step $3$ of Algorithm $1$}
    \label{fig:paths}
\end{figure}{}

\subsection{Comparison Against Other Homotopy Continuation Methods} \label{sec33}
We would like to compare how long it takes to find all solutions of the power flow equations using the new method outlined in Section~\ref{sec31} versus more standard total degree homotopy and polyhedral homotopy methods. The table below gives the average amount of time it takes to find all solutions to the power flow equations in a trial of $100$. We use \texttt{HomotopyContinuation.jl} for all methods \cite{HomotopyContinuation.jl} and do all computations on a $2018$ Macbook Pro with a $2.3$ GHz Quad-Core Intel Core i5 processor.

\begin{table}[h!]
\caption{Average time (seconds) to find all solutions to $K_n$}
\label{timecompletenetworks}
\centering
\begin{tabular}{|c|c|c|c|c|c|c|c|c|}
\hline
$n$ & 4 & 5 & 6 & 7 & 8 &9 \\
\hline
Total Degree & $0.03$ & $0.23$ & $0.97$ & $9.71$ & $46.86$ & $279.38$  \\
\hline
Polyhedral& $0.04$ & $0.29$ & $1.03$ & $18.57$ & $115.25$ & $644.52$  \\
\hline 
Parameter & $0.003$ & $0.03$ & $0.14$ & $0.62$ & $4.85$ & $29.79$  \\
\hline 
\end{tabular}
\end{table}

\begin{table}[h!]
\caption{Average time (seconds) to find all solutions to $C_n$}
\label{timecyclicnetworks}
\centering
\begin{tabular}{|c|c|c|c|c|c|c|c|c|}
\hline
$n$ & 3 & 4 & 5 & 6 & 7 & 8  \\
\hline
Total Degree & $0.01$ & $0.03$ & $0.14$ & $0.70$ & $4.75$ & $23.10$   \\
\hline
Polyhedral& $0.01$ & $0.02$ & $0.10$ & $0.36$ & $2.16$ & $9.60$  \\
\hline 
Parameter & $0.001$ & $0.001$ & $0.01$ & $0.01$ & $0.08$ & $0.13$  \\
\hline 
\end{tabular}
\end{table}

In all cases, we see that parameter homotopy is much faster than polyhedral and total degree homotopy. For the cyclic cases we also see that polyhedral homotopy outperforms total degree homotopy. This agrees with the plot seen in Figure $\ref{fig:paths}$ in that polyhedral homotopy is able to exploit the sparsity present in the cyclic cases and track a fraction of the paths compared with total degree. In contrast, polyhedral homotopy is never better than total degree in the complete cases. This is because the number of paths tracked in polyhedral homotopy is only slightly smaller than in the total degree case. In addition, the start system in the polyhedral case is more time consuming to compute. 

A downside of homotopy methods is that it is possible that not all paths tracked from a start system will make it to a target system. Some reasons for this is that an algorithm could incorrectly conclude a path is diverging to infinity when it is not or that two paths converge to the same solution when they should be distinct. These cases would happen if a solution is large or if two solutions are close together. We experienced these phenomena running our simulations below. If the parameter homotopy step lost solutions, we ran monodromy on the solution set to recover the remaining solutions. This was largely successful and in each topology studied, we found all solutions at least 98.6$\%$ of the time, ensuring accuracy of the computed distributions.

\section{Distributions of the Number of Real Solutions} \label{sec4}
By using the methods developed above, we are able to empirically find distributions of the number of real solutions to the power flow equations much faster, allowing for a more accurate description of the distributions. Using statistical methods, we can be precise about what more accurate means.

Given a random variable $X$, we define its \textit{cumulative distribution function} as 
\begin{align*}
    F(x) &= \mathbb{P}(X \leq x)
\end{align*}{}
for $x \in \mathbb{R}$. Given $n$ independent and identically distributed random variables $X_1,\ldots,X_n$ with cumulative distribution function $F$, we define the \textit{empirical distribution function as}
\begin{align*}
    F_n(x) &= \frac{1}{n} \sum_{i=1}^n \mathbf{1}_{\{X_i \leq x\}}
\end{align*}{}
where $\mathbf{1}$ is the indicator function. $F(x)$ gives the probability that one random variable is less than $x$ where $F_n(x)$ gives the probability that a fraction of random variables is less than $x$. The Dvoretzky–Kiefer–Wolfowitz inequality allows us to give confidence statements about the accuracy of empirical distributions based on the number of samples collected. 
\begin{lem}[Dvoretzky–Kiefer–Wolfowitz Inequality] \cite{dvoretzky1956asymptotic} \label{DKWinequality}

With probability $1- \alpha$
\begin{align}
    F_n(x) - \epsilon \leq F(x) \leq F_n(x) + \epsilon
\end{align}{}
where $\epsilon = \sqrt{\frac{\ln{\frac{2}{\alpha}}}{2n}}$.
\end{lem}{}

This provides a way to assess the accuracy of our empirical results with high probability. For all distributions we evaluate the number of real solutions of the power flow equations on at least $1.4$ million samples, implying by Lemma \ref{DKWinequality} that with $99\%$ probability, the true cumulative distribution function is within $\epsilon = 0.0005$ of what is listed.

In addition to computing empirical distributions, we would like to visualize these distributions via the space of susceptances. \\
We consider a \textit{solution region} to be a region in the space of susceptances where the number of real solutions is fixed. This notion is analogous to \textit{cylindrical algebraic decomposition} used in computer vision and real algebraic geometry \cite{caviness2004quantifier}. Algorithms for computing these decompositions exist, but become computationally unattainable for networks on more than a few nodes. Instead we sample susceptances on a unit hyperphere and count the number of real solutions to the power flow equations for these values. We then assign a color to each number of real solutions and we color susceptances on the sphere according to this scheme.

\subsection{Cyclic Networks}\label{sec41}

The distribution of the number of nontrivial real solutions for $C_3$ was completely solved and for $C_4$ was closely analyzed in \cite{zachariah2018distributions} by using Mathematica to symbolically solve the entire system. The authors proved that the distribution for $C_3$ is given by
 \begin{align}
     \mathbb{P}(\text{number of nontrivial real solutions}=0) &= 3 - \frac{4}{\sqrt{3}} \approx 0.6906 \\
     \mathbb{P}(\text{number of nontrivial real solutions}=2) &= \frac{4}{\sqrt{3}}-2 \approx 0.3904,
 \end{align}{}
 and the distribution for $C_4$ is 
  \begin{align}
     \mathbb{P}(\text{number of nontrivial real solutions}=0) &\approx 0.6945 \\
     \mathbb{P}(\text{number of nontrivial real solutions}=4) &\approx 0.3055.
 \end{align}{}

 Figure $\ref{C3solregions}$ shows the distribution for $C_3$ in the space of susceptances where blue regions are where there are no nontrivial real solutions and red regions are where there are $2$ nontrivial real solutions.
\begin{figure}[h!]
    \centering
    \includegraphics[width =0.3\textwidth]{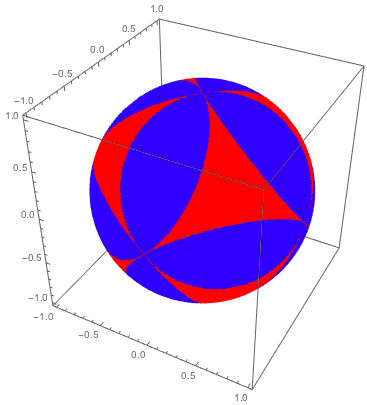}
    \caption{Solution regions of $C_3$}
    \label{C3solregions}
\end{figure}{}

For any graph with more than three edges we need to fix all but three of them in order to visualize the solution regions. Figure \ref{C4solregions} shows two examples for $C_4$ where $b_{01}$ is fixed. 

\begin{figure}
\centering    
\subfigure{\includegraphics[width=0.2\textwidth]{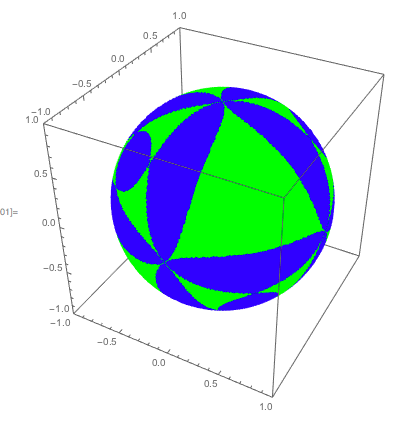}} 
\subfigure{\includegraphics[width=0.2\textwidth]{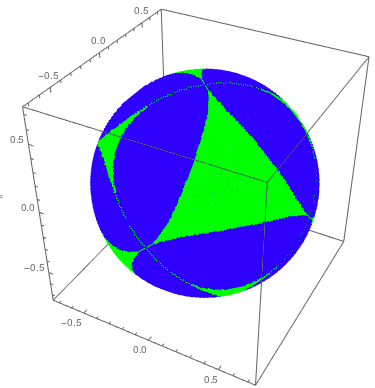}} 
\caption{Solution Regions for $C_4$ with $b_{01}=0.1$ (left) and $b_{01}=0.3$ (right)}
\label{C4solregions}
\end{figure}


We can also visualize solution regions for $C_5$ after fixing two of the susceptances. Examples of this are given in Figure \ref{C5}. In these images we observe a lot of symmetry; this can be explained as the number of real solutions to cyclic networks is unchanged under any permutation of the edges, so also of the susceptances. The color scheme for solution regions of all pictures is given in Table \ref{colorguide}.

\begin{table}[h!]
\caption{Colors of solution regions}
\label{colorguide}
\centering
\begin{tabular}{|c|c|c|c|c|c|}
\hline
$\#$ of Nontrivial Real Solutions & 0&2&4&6\\
Color & \text{Blue} & \text{Red} &\text{Green} &\text{Purple}  \\
\hline 
$\#$ of Nontrivial Real Solutions &8&10&12&14 \\
Color &\text{Yellow} &\text{Black} &\text{Orange} &\text{Pink} \\
\hline 
\end{tabular}
\end{table}

\begin{figure}
\centering    
\subfigure{\includegraphics[width=0.2\textwidth]{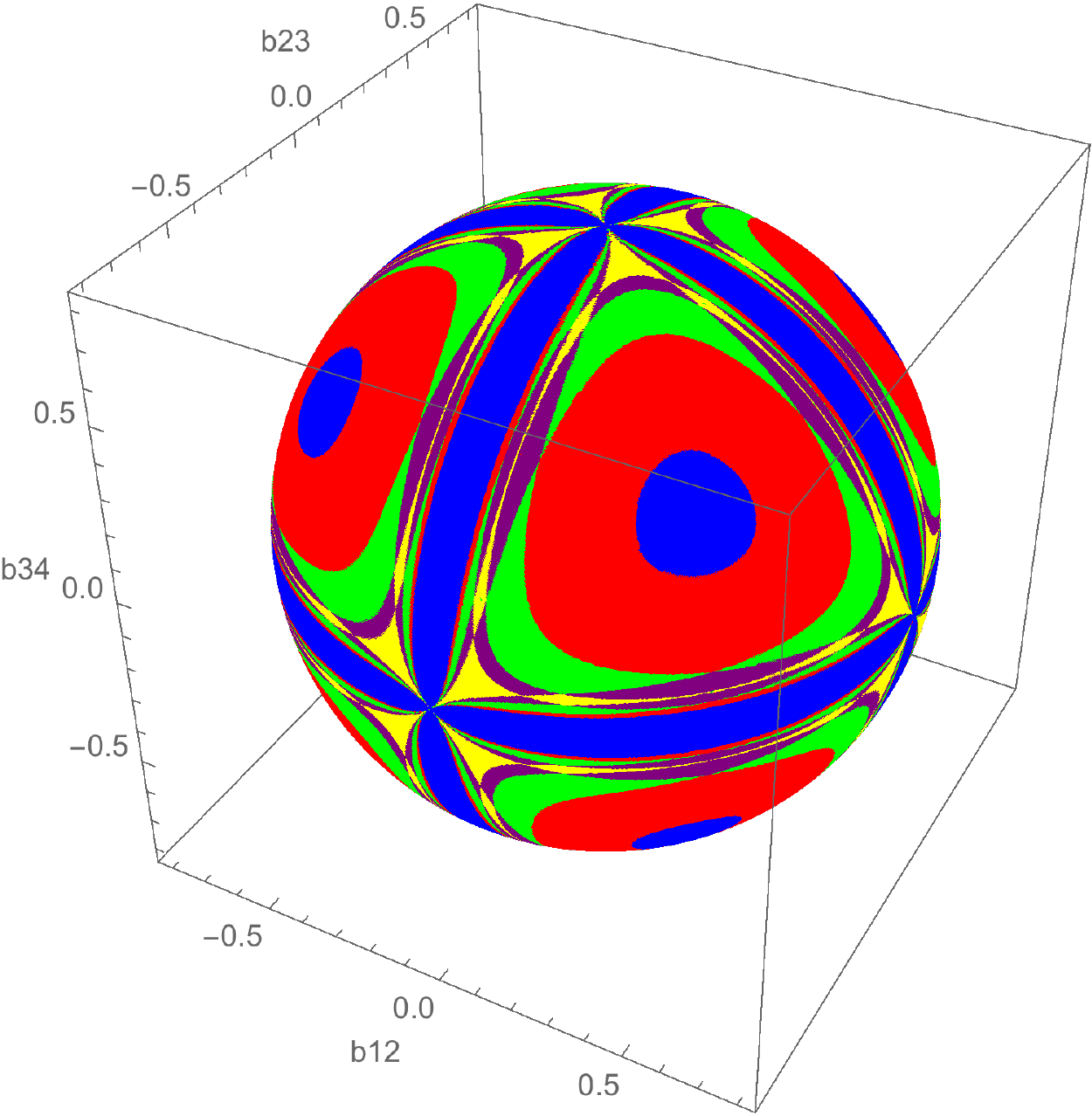}} 
\subfigure{\includegraphics[width=0.2\textwidth]{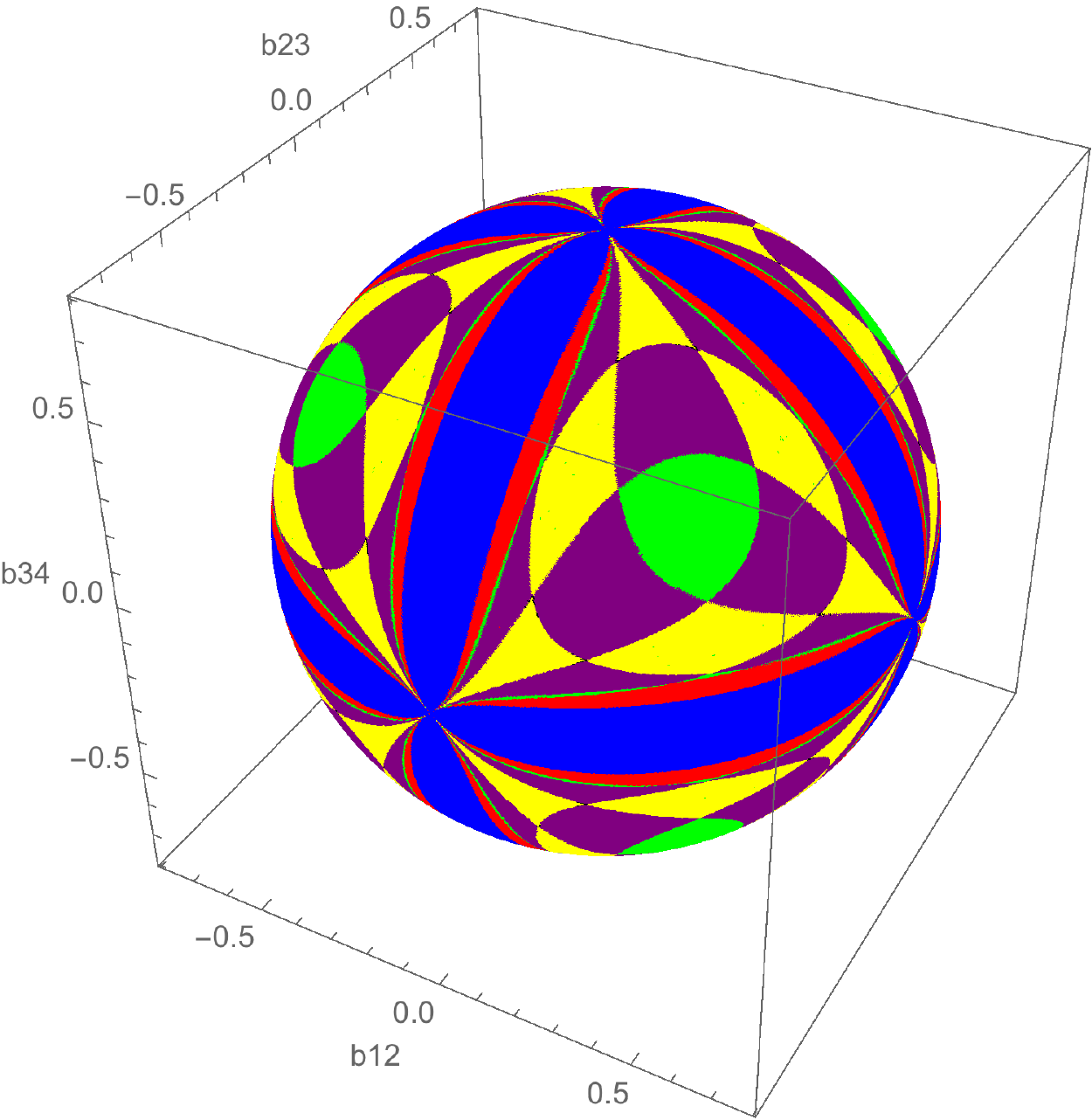}} 
\caption{Solution regions for $C_5$ with $b_{01}=0.5$, $b_{04}=0.3$ (right) and $b_{01}=0.6$, $b_{04}=0.2$ (left)}
\label{C5}
\end{figure}

\begin{figure}
\centering    
\subfigure{\includegraphics[width=60mm]{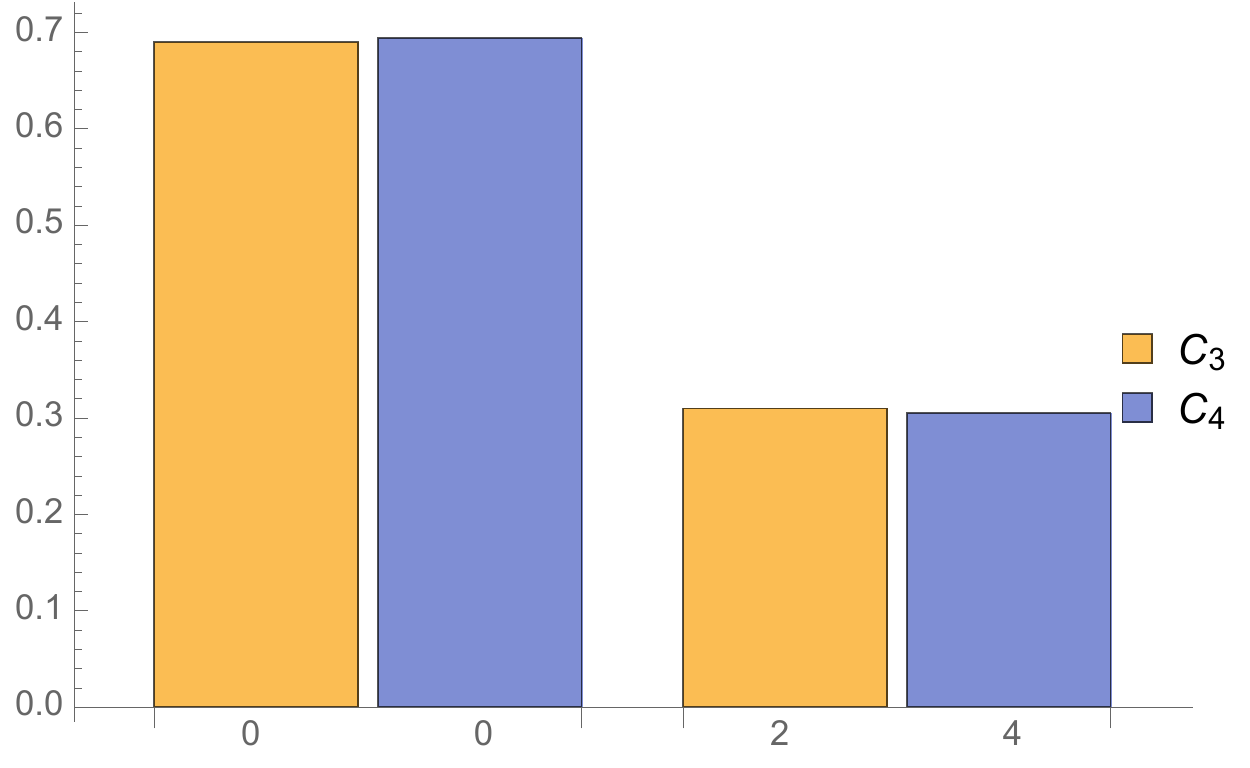}} 
\subfigure{\includegraphics[width=60mm]{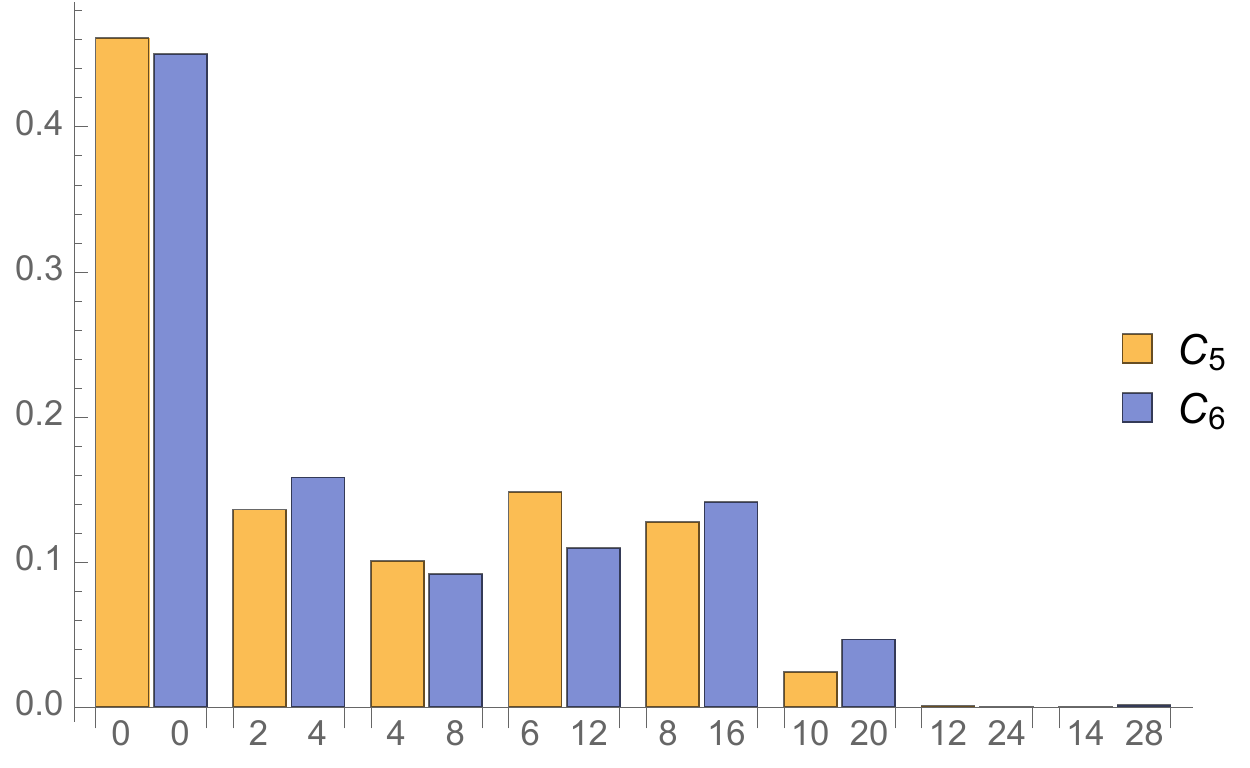}} 
\subfigure{\includegraphics[width=60mm]{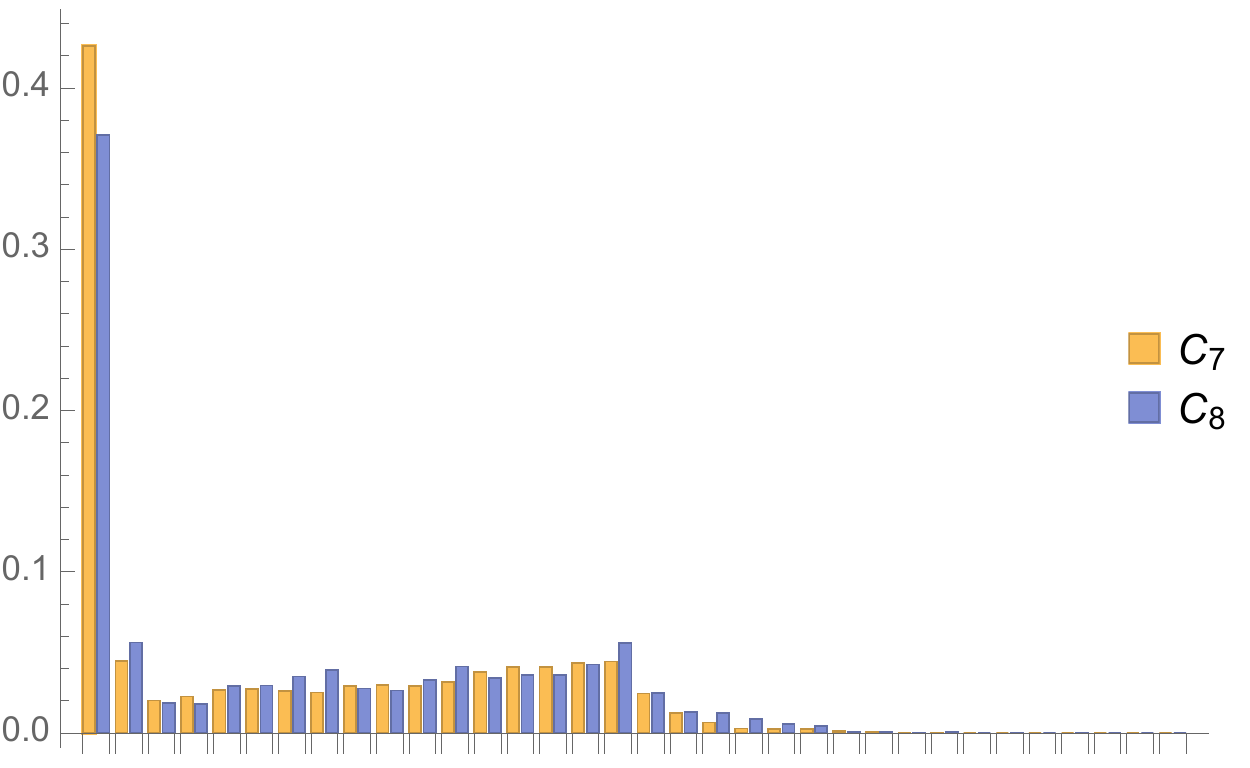}} 
\subfigure{\includegraphics[width=60mm]{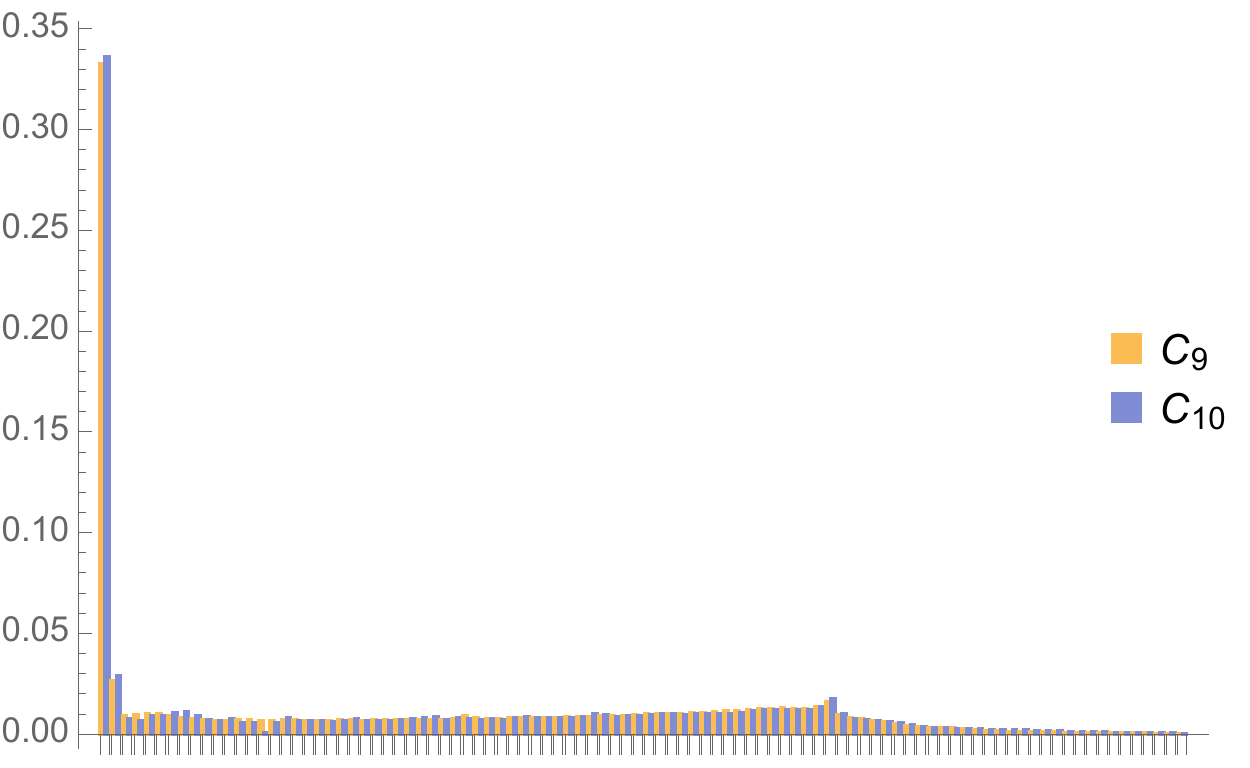}} 
\caption{Distribution of number of nontrivial real solutions for cyclic networks}
\label{cyclic_distributions}
\end{figure}

 Numerical results for $C_6,C_7,C_8,C_9$ and $C_{10}$ are given in Appendix \ref{appendixdata} in Tables~\ref{c6dist}-\ref{c10dist} and are shown graphically in Figure \ref{cyclic_distributions}. We graph the distributions for $C_n,C_{n+1}$ next to each other for $n \in \{3,5,7,9\}$ since the support for $C_{n+1}$ is that of $C_n$ scaled by two. For all cyclic cases we notice a major left skew in the distribution. In addition, we notice that $C_5,\ldots,C_{10}$ are multimodal and $C_n$ and $C_{n+1}$ seem to have similar numbers of modes, although they occur in different places.

In addition, for $C_3,\ldots,C_6$ we find susceptance values where each system attains the maximal number of  real solutions, $n \binom{n-1}{\lfloor \frac{n-1}{2}\rfloor }$. We can generalize this for all cycles.

\begin{theorem} \label{cyclicmaxreals}
 $C_n$ for all $n\geq 3$ has susceptance values that achieve the generic maximum bound of $n \binom{n-1}{\lfloor \frac{n-1}{2}\rfloor }$ real solutions.
\end{theorem}{}
\begin{proof}
First consider the case of $C_n$ where $4 \nmid n$. Set all susceptances equal to $1$. The system of equations defined in \eqref{trigpfeqs1} with $|V_k| = 1, g_{km} = 0$ becomes 
\begin{align}
    \sin(\theta_k - \theta_{k-1}) &= \sin(\theta_{k+1}-\theta_k) \label{equalbij}
\end{align}
for $k=1,\ldots,n$. Consider the change of variables $u_k = \frac{\theta_k - \theta_{k-1}}{\pi}$ for $k=1, \ldots, n$. This transforms \eqref{equalbij} into
\begin{align}
    \sin(\pi u_k) &= \sin( \pi u_{k+1}) \\
    \sum_{k=1}^n u_k &\equiv 0 \mod 2 \label{summod2}
\end{align}
for $k=1,\ldots, n$. This means that for all $k,m$, $u_k = u_m$ or $u_k = 1-u_m$. This allows us to partition the set $S = \{u_1,\ldots,u_n\}$ into two sets: $S_1 = \{ u_k \in S : u_k = 1-u\}$ and $S_2 = \{u_k \in S : u_k = u\}$ for some $u \in \mathbb{R}$. Let $|S_1| = m$ and $|S_2| = n-m$. 

By \eqref{summod2} we have $(n-2m)u + m \equiv 0 \mod 2$. Suppose $n-m$ is odd. This gives $(n-2m)u \equiv 1 \mod 2$. There are $n-2m$ different $u$ that satisfy this, namely $u = \frac{s}{n-2m}$ for $s \in \{1,3,5,\ldots, 2(n-2m)-1\}$. Now suppose $n-m$ is even. By \eqref{summod2} we have $(n-2m)u \equiv 0 \mod 2$. There are $n-2m$ solutions to this equation, namely $u = \frac{s}{n-2m}$ where $s \in \{0,2,4,\ldots, 2(n-2m-1)\}$. In either case, there are $\binom{n}{m}$ ways to construct $S_1$, giving $\binom{n}{m}(n-2m)$ such solutions $u$ for each $m \leq k$. This gives $\sum_{m=0}^k \binom{n}{m}(n-2m) = (k+1) \binom{n}{k+1} = n \binom{n-1}{k}$ real solutions where the first equality is $(5.18)$ of \cite{graham1994concrete} and the second equality is $(1.2)$ of \cite{combinatorialidentities}. 


Now consider $C_n$ for $n = 4k$ for $k\in \mathbb{N}$. As per the proof of \ref{cmod4infinite}, the previous choice of susceptances produces infinitely many solutions. So instead, consider susceptances $b_{01} = -1$ and $b_{ij} = 1$ for all other edges $ij$. Using the same notation as above, we have that $u_1 = -u$ or $u_1 = 1+u$ and $u_k = u$ or $u_k = 1-u$ for all $2 \leq k \leq n$. For all $k \geq 2$, let $S_1 = \{u_k \in S \backslash u_1 : u_k = 1-u \}$ and $S_2 = \{ u_k \in S \backslash u_1 : u_k = u \}$. Note that for $u_1 = -u$, $|S_1| = m$, and $|S_2| = n-m-1$ \eqref{summod2} gives $(2+2m-n) u \equiv m \mod 2$. Similarly, for $u_1 = 1+u, |S_1| = n-m-1$ and $|S_2| = m$ \eqref{summod2} also gives $(2+2m-n) u \equiv m \mod 2$ so the solutions to the two cases are redundant. Therefore, without loss of generality we suppose $u_1 = -u, |S_1| = m$ and $|S_2| = n-m-1$. When $m$ is odd \eqref{summod2} gives $(2+2m-n) u \equiv 1 \mod 2$. This equation has $|2+2m-n|$ solutions $\mod 2$, namely $u = \frac{s}{2+2m-n}$ for $s \in \{1,3,5,\ldots, 2(2+2m-n)-1\}$. When $m$ is even we want to find all solutions to $(2+2m-n) u \equiv 0 \mod 2$. Again, this equation has $|2+2m-n|$ solutions $\mod 2$, namely $u = \frac{s}{2+2m-n}$ for $s \in \{0,2,4,\ldots, 2(2m-n+1)\}$. For each $m \leq n-1$ there are $\binom{n-1}{m}$ ways to construct $S_1$. This gives a total of $\sum_{m=0}^{n-1} \binom{n-1}{m}|2+2m-n| = n \binom{n-1}{k-1}$ real solutions where the equality is proven below in Lemma~\ref{binomequality}.
\end{proof}

\begin{lem} \label{binomequality}
 $\sum_{m=0}^{2k-1} \binom{2k-1}{m}|2+2m-2k| = 2k\binom{2k-1}{k-1}$.
\end{lem}
\begin{proof}
Without the absolute value, we see $\sum_{m=0}^{2k-1} \binom{2k-1}{m}|2+2m-2k|$ is equal to
\begin{align}
& \sum_{m=0}^{k-1}\binom{2k-1}{m}(2k-2m-2) \label{binom1}\\ + \  &\sum_{m=k-1}^{2k-1} \binom{2k-1}{m}(2+2m-2k) \label{binom2}
\end{align}
Applying $(5.18)$ of \cite{graham1994concrete} to \eqref{binom1} we see that \eqref{binom1} is equal to
\begin{align}
 k \binom{2k-1}{k}- \sum_{m=0}^{k-1}\binom{2k-1}{m} = k \binom{2k-1}{k-1}- \sum_{m=0}^{k-1}\binom{2k-1}{m} \label{identity1}
\end{align}
Again, using $(5.18)$ of \cite{graham1994concrete} we get the identity
\begin{align}
    \sum_{m=0}^{2k-1}\binom{2k-1}{m}(2k-1-2m) &= 2k\binom{2k-1}{2k} = 0
\end{align}
This gives the identity for \eqref{binom2} as 
\begin{align}
 \sum_{m=0}^{2k-1}\binom{2k-1}{m} - \sum_{m=0}^{k-1}\binom{2k-1}{m}(2+2m-2k) \label{identity2}
\end{align}
Adding \eqref{identity1} and \eqref{identity2} we see
\begin{align}
    \sum_{m=0}^{2k-1} &\binom{2k-1}{m}|2+2m-n| = k\binom{2k-1}{k-1}-\sum_{m=0}^{k-1}\binom{2k-1}{m}\\&+\sum_{m=0}^{2k-1}\binom{2k-1}{m} - \sum_{m=0}^{k-1}\binom{2k-1}{m}(2+2m-2k) \label{binom3}
\end{align}
Applying \eqref{identity1} to the last term in \eqref{binom3} gives
\begin{align*}
    \sum_{m=0}^{2k-1} &\binom{2k-1}{m}|2+2m-n| = 2k \binom{2k-1}{k-1} \\ &- 2 \sum_{m=0}^{k-1}\binom{2k-1}{m}
    +\sum_{m=0}^{2k-1}\binom{2k-1}{m} \\
    &= 2k \binom{2k-1}{k-1} - \sum_{m=0}^{k-1}\binom{2k-1}{m}
    +\sum_{m=k}^{2k-1}\binom{2k-1}{m} \\
    &= 2k \binom{2k-1}{k-1}
\end{align*}
\end{proof}

Since in our trial of $1.4$ million samples we did not sample any susceptances for $C_7-C_{10}$ that gave the maximum number of real solutions, we suspect that such susceptance values occur with very small probability (but they do occur as our proof is constructive). In the final case of Theorem \ref{cyclicmaxreals} we switch from the susceptances being equal. The following lemma explains why.

\begin{lem} \label{cmod4infinite}
 There exist susceptance values for $C_n$, $4 \mid n$, where there are infinitely many real solutions.
\end{lem}{}
\begin{proof}
Set all susceptances equal to $1$. The system of equations as defined in $(\ref{trigpfeqs1})$ with $|V_k|=1, g_{km} = 0$ becomes for $k=1,\ldots, n$

\begin{align}
    \sin(\theta_k - \theta_{k-1}) &= \sin(\theta_{k+1}-\theta_k)
\end{align}{}

where the indices wrap around $\mod n$. Set $u_k = \frac{\theta_k-\theta_{k-1}}{\pi}$ for $k=1\ldots, n$. Under these  coordinates we know that $\sum_{k=1}^{n} u_k \equiv 0 \mod 2$ and $u_k=u_l$ or $u_k = 1-u_l$ for all $k,l$. Now partition the set $\{u_1,\ldots,u_n\}$ into two equal size sets $S_1$ and $S_2$. For all $u_k \in S_1$ set $u_k = u$ for some $u \in \mathbb{R}$. For all $u_k \in S_2$ set $u_k = 1-u$. This implies that $\sum_{k=1}^{n} u_k = \frac{n}{2}u + \frac{n}{2}(1-u) = \frac{n}{2} \equiv 0 \mod 2$ satisfying the first condition. Since we can choose any $u \in \mathbb{R}$, this implies that there are infinitely many real solutions. 
\end{proof}{}

By \cite{mallada2016distributed} we know that with probability one there are finitely many solutions to the power flow equations. This implies that such susceptances where $C_n$ admits infinitely many real solutions lie on a set of measure zero and are not generic.

\subsection{Complete Networks}\label{sec42}
We perform a similar analysis to section \ref{sec41} but this time on complete networks on $4,5,6,7$ and $8$ vertices. The results of these simulations are given in Appendix \ref{appendixdata} in Tables~\ref{k4dist}-\ref{k7dist} and shown graphically in Figure \ref{K4K7dist}.

\begin{figure}
    \centering
    \includegraphics[width =0.5\textwidth]{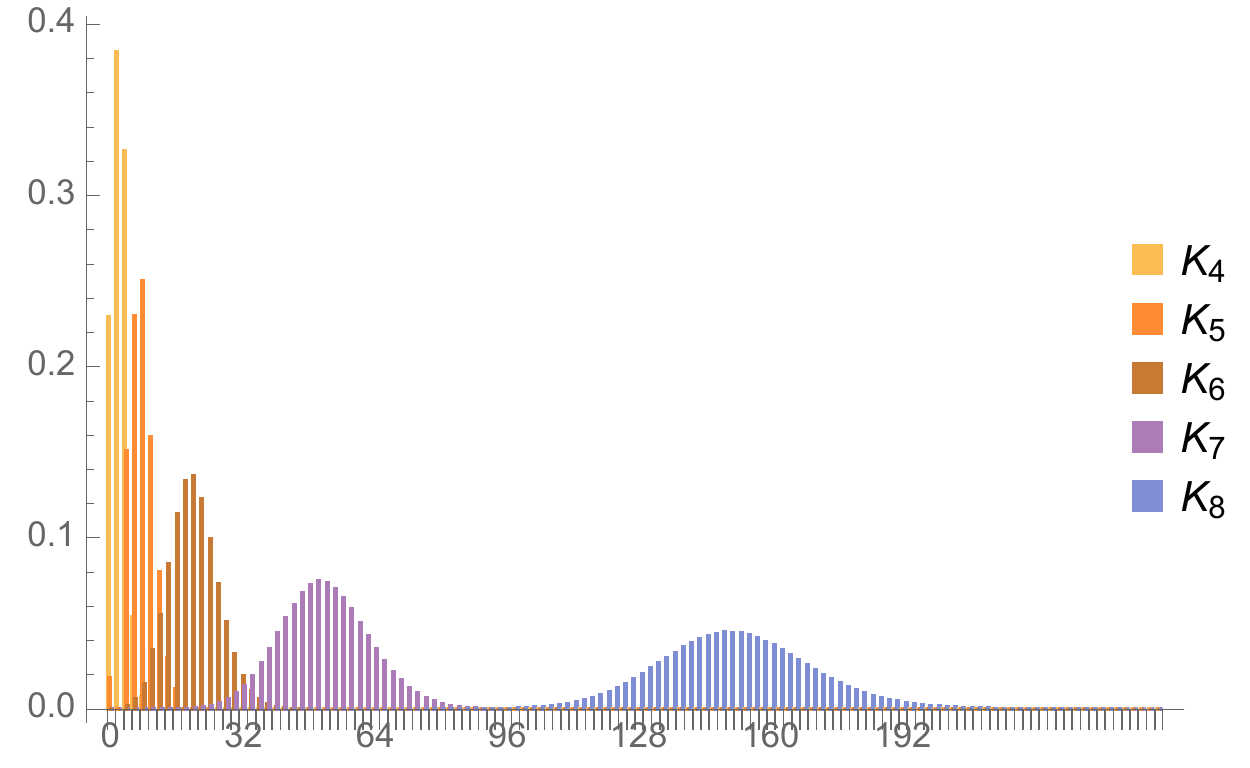}
    \caption{Distribution of the number of nontrivial real solutions for $K_4,K_5,K_6,K_7,K_8$}
    \label{K4K7dist}
\end{figure}{}

In contrast to  the cyclic networks we see the distributions for the complete networks tending to more of a normal shape. While they are still left skewed compared to the range given by the complex bound, there isn't as large of a number of instances with zero nontrivial real solutions. We also see that as $n$ increases the variance becomes much larger and the curve flattens. A major open question for this family of graphs is the following:
\begin{question}
What is the maximum number of real solutions for a complete $n-$node network?
\end{question}{}

We can use the Ballieul and Byrnes bound \cite{baillieul1982geometric} to get an upper bound of $\binom{2n-2}{n-1}$. For $n=3$ this bound is achievable as we see instances of $6$ real solutions. For $n=4$, this bound says we have at most $20$ real solutions but so far we have only found examples of at most $18$  real solutions. It is open as to whether or not this bound is tight for any case $n \geq 4$.

We also study the solution regions for $K_4$ in Figure \ref{K4}. We observe similarities in the shape of some of the solution regions in all cases. In each case there appear to be almost convex, quasi polygonal areas. In contrast to the cyclic cases we don't observe any symmetry. This is explained as the only automorphism of $K_4$ that fix edges $e_{01}, e_{02}$ and $e_{03}$ is the identity mapping.

\begin{figure}
\centering    
\subfigure{\includegraphics[width=0.2\textwidth]{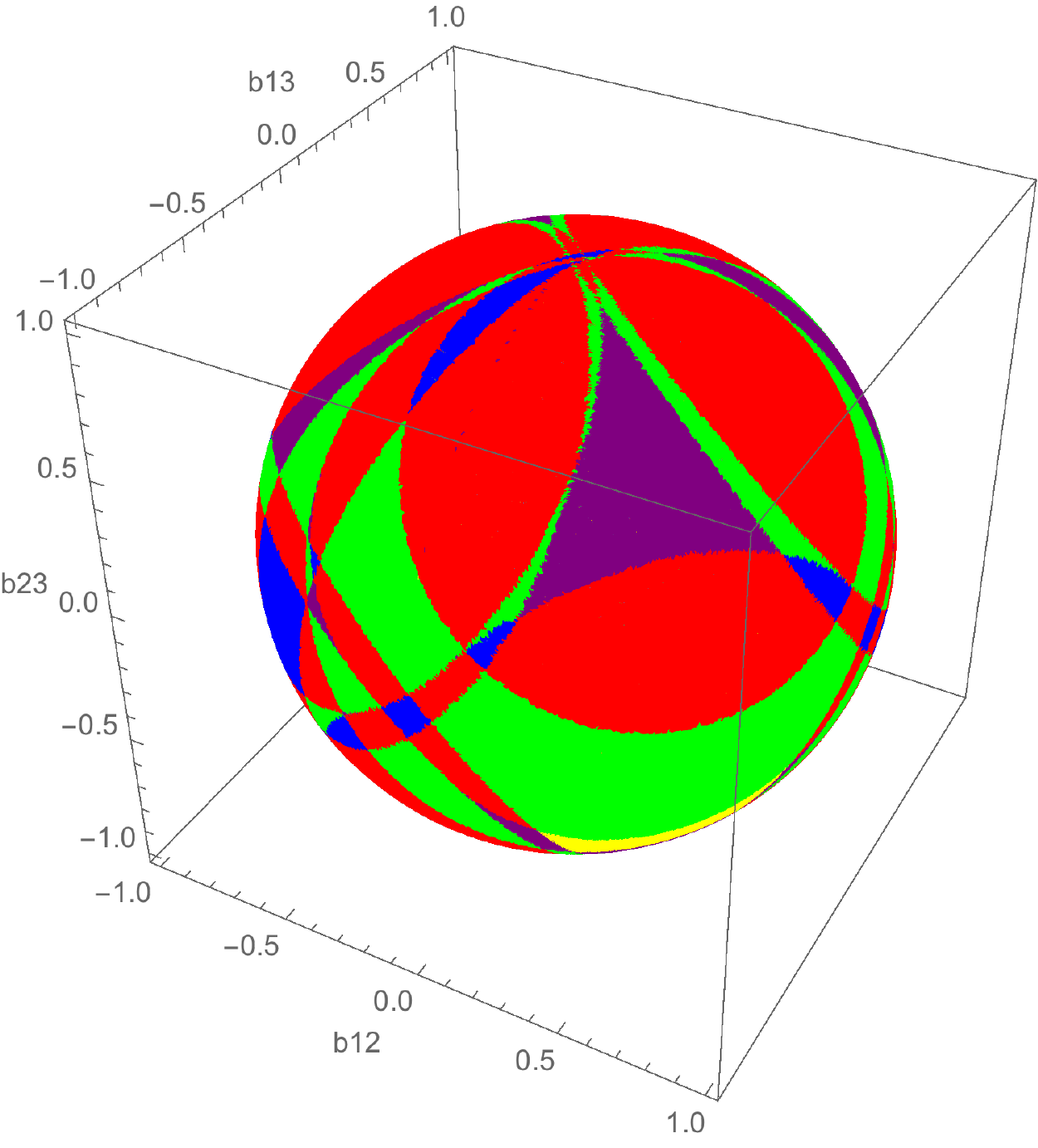}} 
\subfigure{\includegraphics[width=0.2\textwidth]{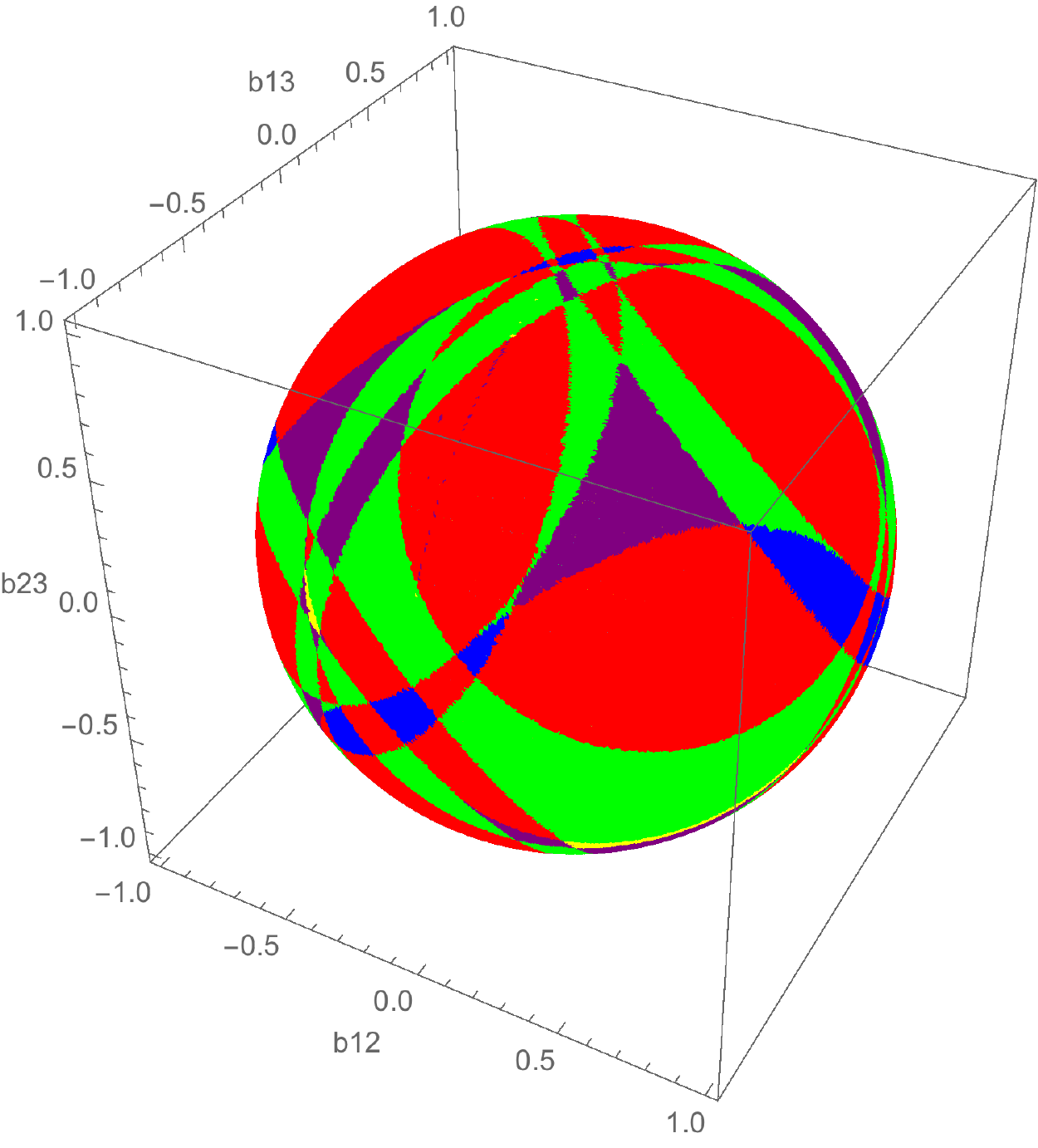}} 
\caption{Solution Regions for $K_4$ with $b_{01}=0.03$, $b_{02}=0.15$, $b_{03} = 0.2$ (right) and $b_{01}=0.1$, $b_{02}=0.2, b_{03} = 0.3$ (left)}
\label{K4}
\end{figure}

\subsection{Number of Real Solutions to Random Polynomials}         
Much of this work was motivated by the observation that the power flow equations generically admit few real solutions compared to the complex bounds. This has been well documented in existing power systems literature \cite{mehta2016numerical,ma1993an}. While we agree that the number of real solutions tends to be low when compared with the total number of complex solutions, we observe that when compared with a random polynomial system, the power flow equations actually admit more real solutions than should be expected! In \cite{rouillier1999solving} it is shown that finding the distribution of the number of real solutions to a system of polynomial equations is the same as finding the distribution to that of a univariate polynomial whose coefficients are polynomials in the coefficients of the original polynomial system. Since we can reduce the distribution of the number of real solutions to the power flow equations to a single univariate polynomial, a natural question then arises.
\begin{question}
How does the distribution of the number of nontrivial real solutions to the power flow equations with $N$ complex solutions compare to that of a random univariate polynomial of degree $N$? 
\end{question}

We compare the distribution of the power flow equations with $N$ nontrivial complex solutions to that of $q(x) = \sum_{i=0}^N c_i x^i$ where $c_i \sim \mathcal{N}(0,1)$. 

Kac's formula \cite{kac1943on} gives a closed form for the expected number of real roots of $q(x)$ as
\begin{align}
    \frac{1}{\pi}\int_{-\infty}^\infty \sqrt{\frac{1}{(t^2-1)^2}-\frac{(N+1)^2t^{2N}}{(t^{2N+2}-1)^2}}dt. \label{kacsformula}
\end{align}{}
    
We can numerically approximate this integral to get the expected number of real solutions for $q$. In the results below, we use \eqref{kacsformula} to get an approximation for the expected number of real roots of $q$ in each case.

\subsubsection{Cyclic Networks} \label{sec431}
 We compare the distribution of the number of real solutions to cyclic $3-10$ node networks to that of random polynomials of corresponding degrees.

 We run numerical simulations using the $\mathtt{CountRoots}$ function in Mathematica to compute $10,000$ trials for $q$. We use (\ref{kacsformula}) to calculate the expected number of real roots for $q$ in Table \ref{t7}.
 
Figure $4$ plots the distributions of the number of nontrivial real solutions for cyclic networks against that of real roots corresponding to random polynomials of appropriate degree. We see that while the cyclic networks seem to give many more instances of zero nontrivial real solutions than random polynomials give of zero real roots, there also seems to be a much higher chance of getting instances of larger numbers of nontrivial real solutions with cyclic networks than with that of a random polynomial. This phenomenon is reflected in Table~\ref{t7} as we see the expected number of nontrivial real solutions  is higher than that of random polynomials for $C_5,\ldots,C_{10}$. We suspect that this is true for all $C_n$, $n \geq 5$, and the gap between the two values will continue to increase.

\begin{table}
\caption{Expected Number of Nontrivial Real Solutions to Cyclic Networks}
\label{t7}
\centering
\begin{tabular}{||c|c||c|c||c|c||c|c||}
\hline
$C_3$ & $0.62$ & $C_4$ & $1.22$ & $C_5$ & $2.85$ & $C_6$ & $5.93$  \\
$q(x)$ & $1.30$  & $q(x)$ & $1.64$ & $q(x)$ & $2.35$ & $q(x)$ & $2.77$  \\
\hline
 $C_7$ & $11.57$ & $C_8$ &$25.57$ & $C_9$ & $52.38$ & $C_{10}$ & $105.40$ \\
$q(x)$ & $2.96$ & $q(x)$ & $3.83$ & $q(x)$ & $4.40$ & $q(x)$ & $4.84$  \\
\hline 
\end{tabular}
\end{table}

\begin{figure}
\centering    
\subfigure{\includegraphics[width=0.2\textwidth]{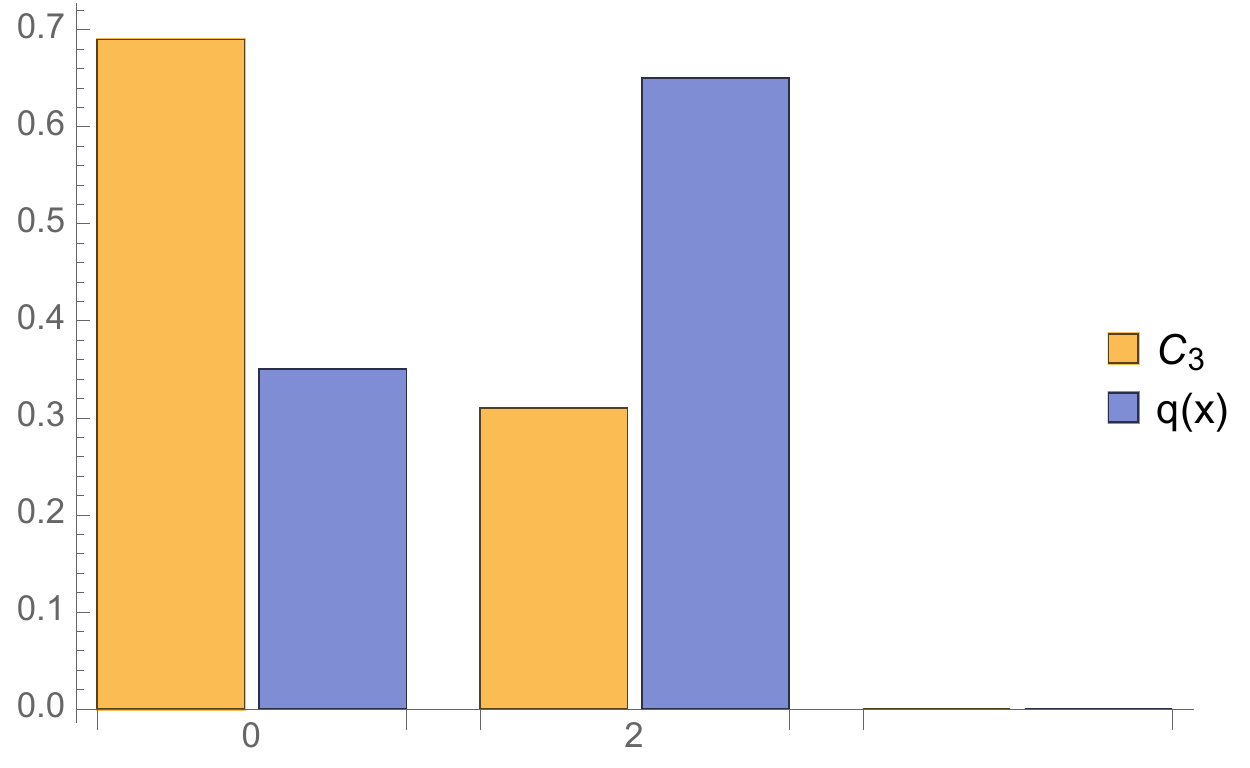}}
\subfigure{\includegraphics[width=0.2\textwidth]{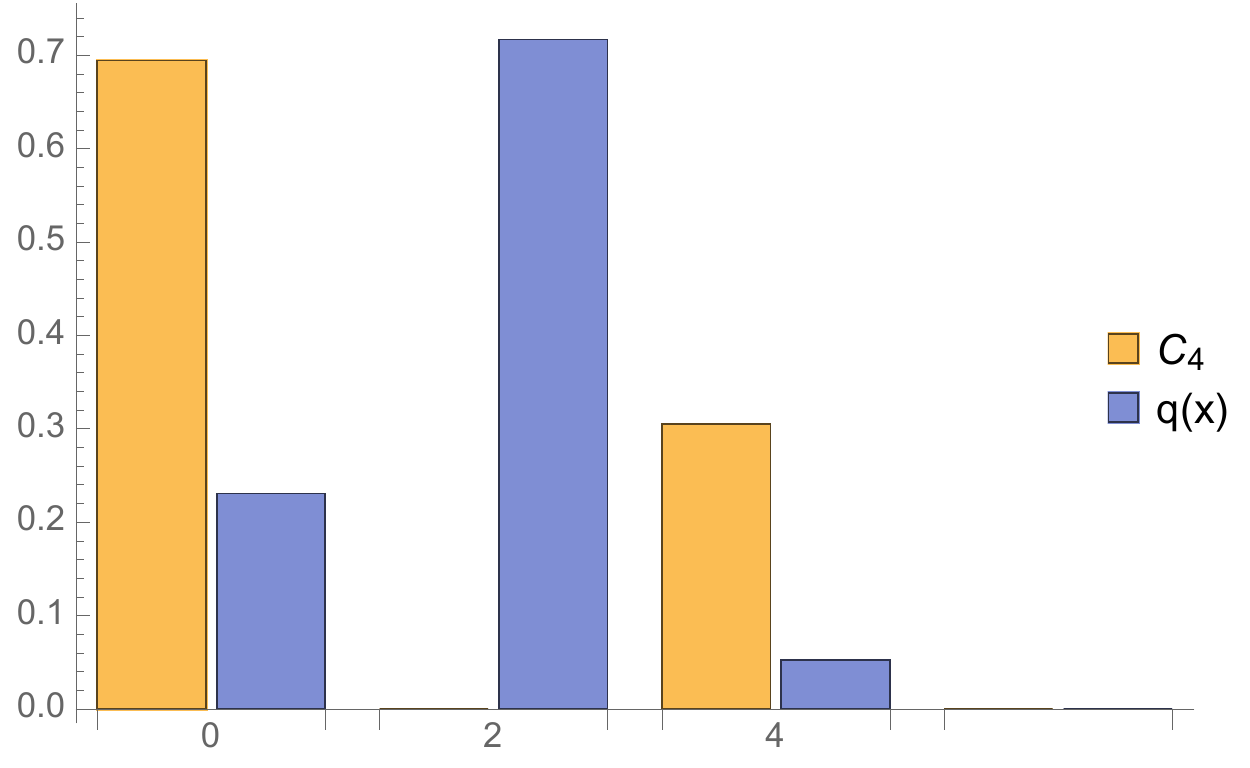}}
\subfigure{\includegraphics[width=0.2\textwidth]{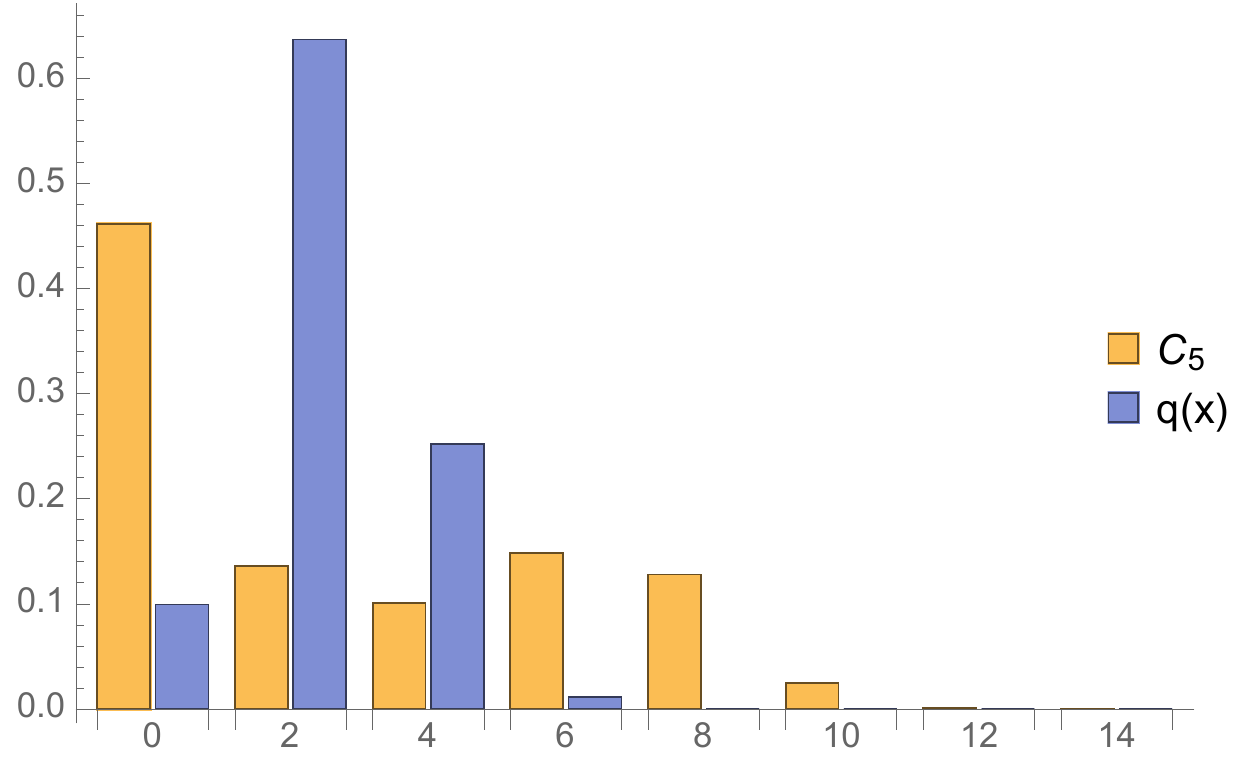}}
\subfigure{\includegraphics[width=0.2\textwidth]{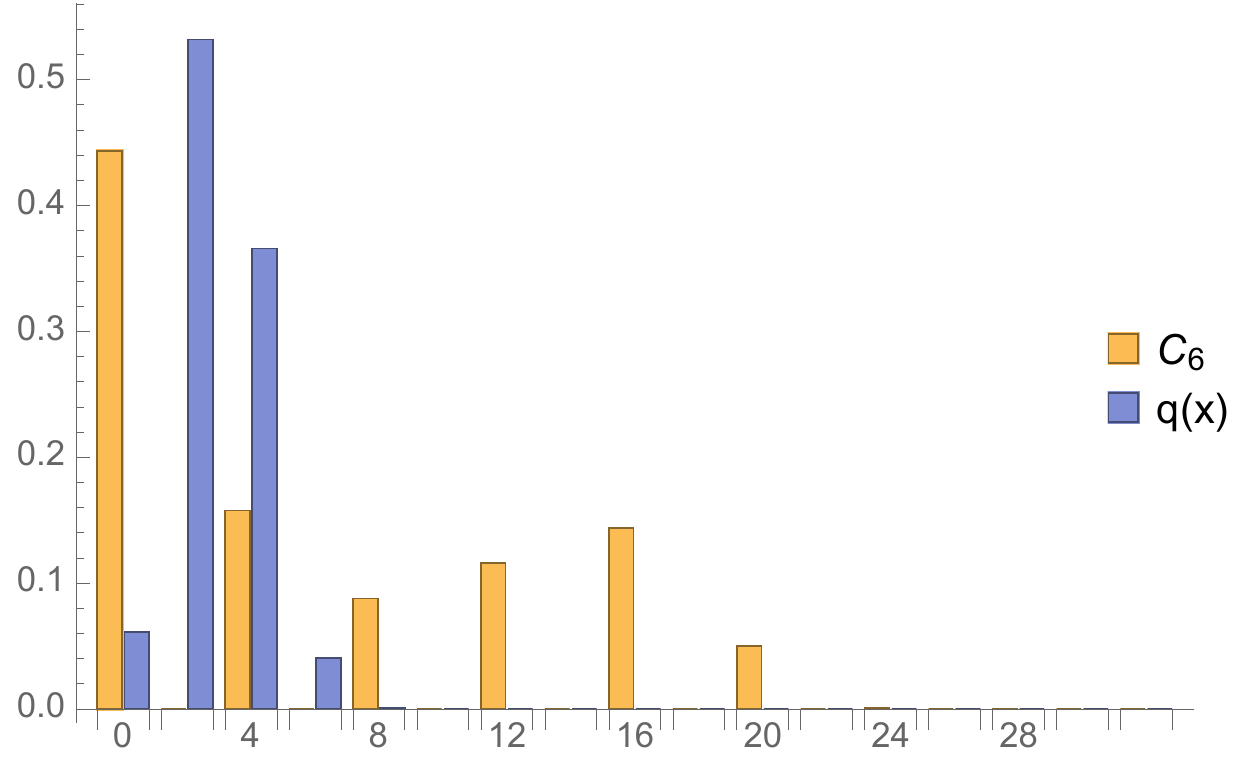}}
\subfigure{\includegraphics[width=0.2\textwidth]{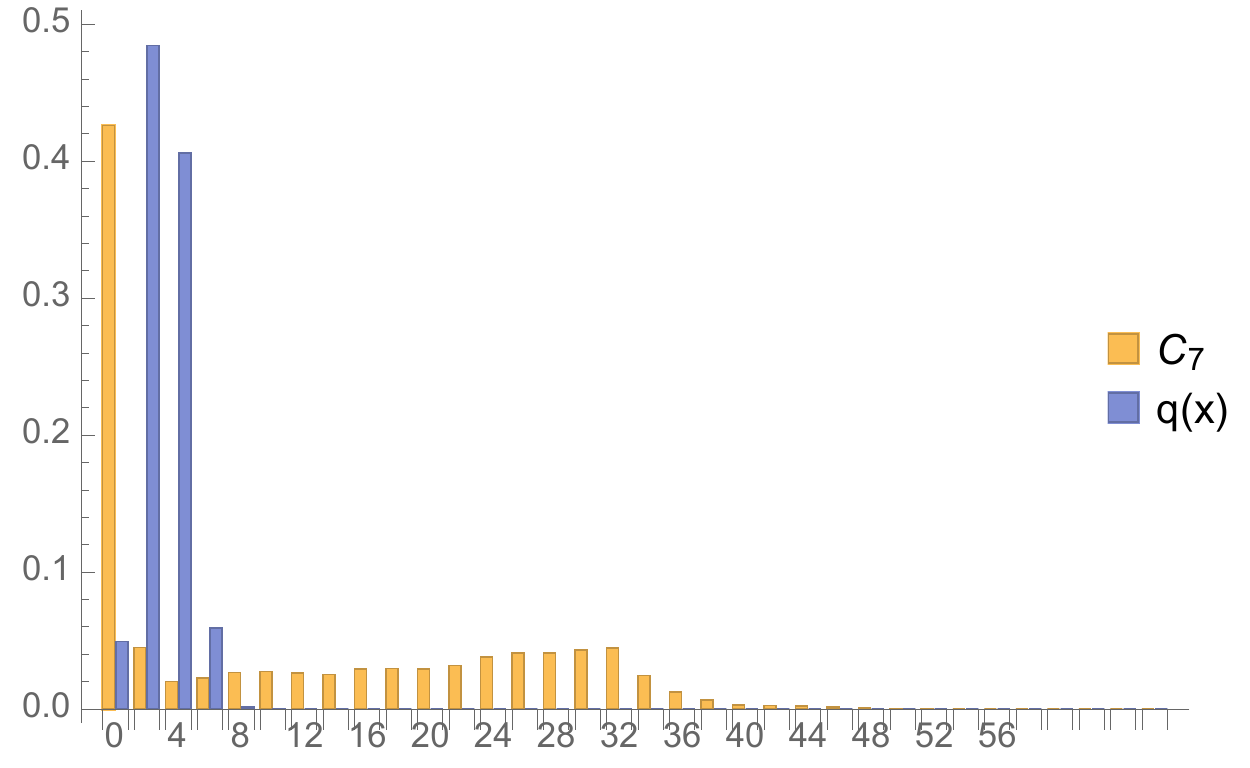}}
\subfigure{\includegraphics[width=0.2\textwidth]{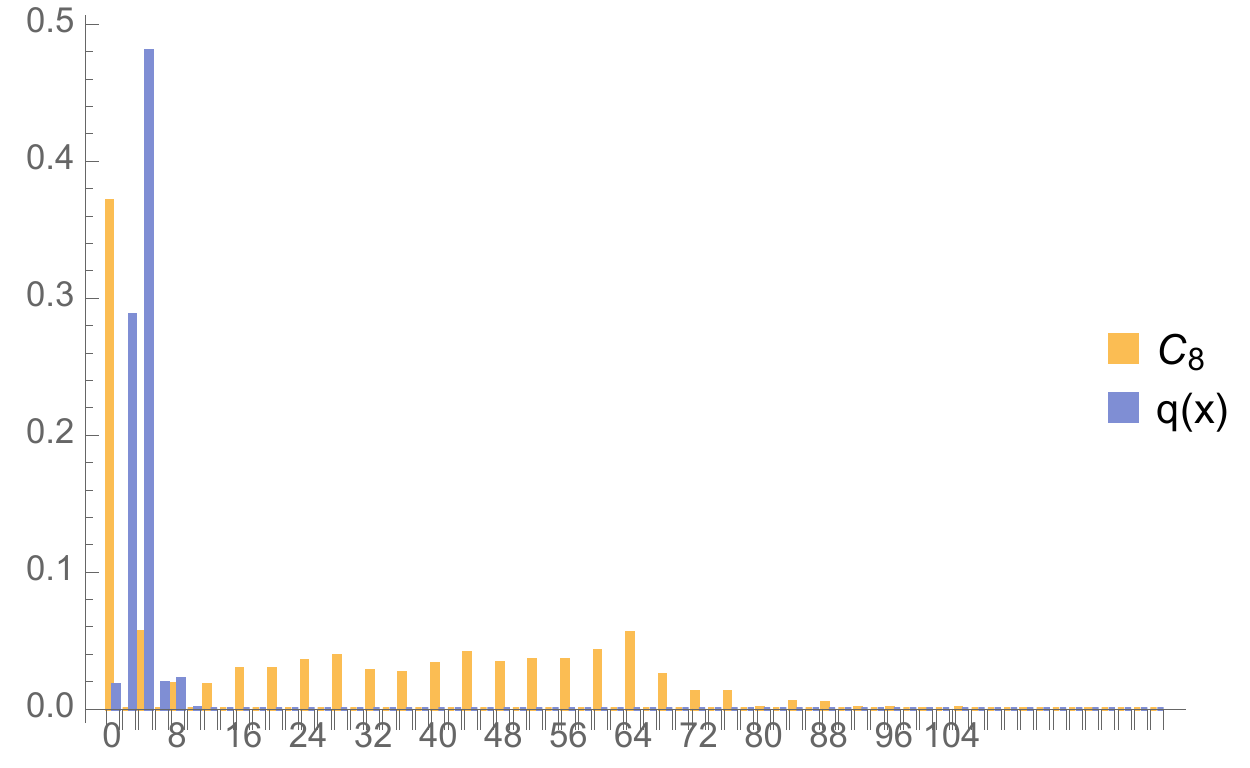}}
\subfigure{\includegraphics[width=0.2\textwidth]{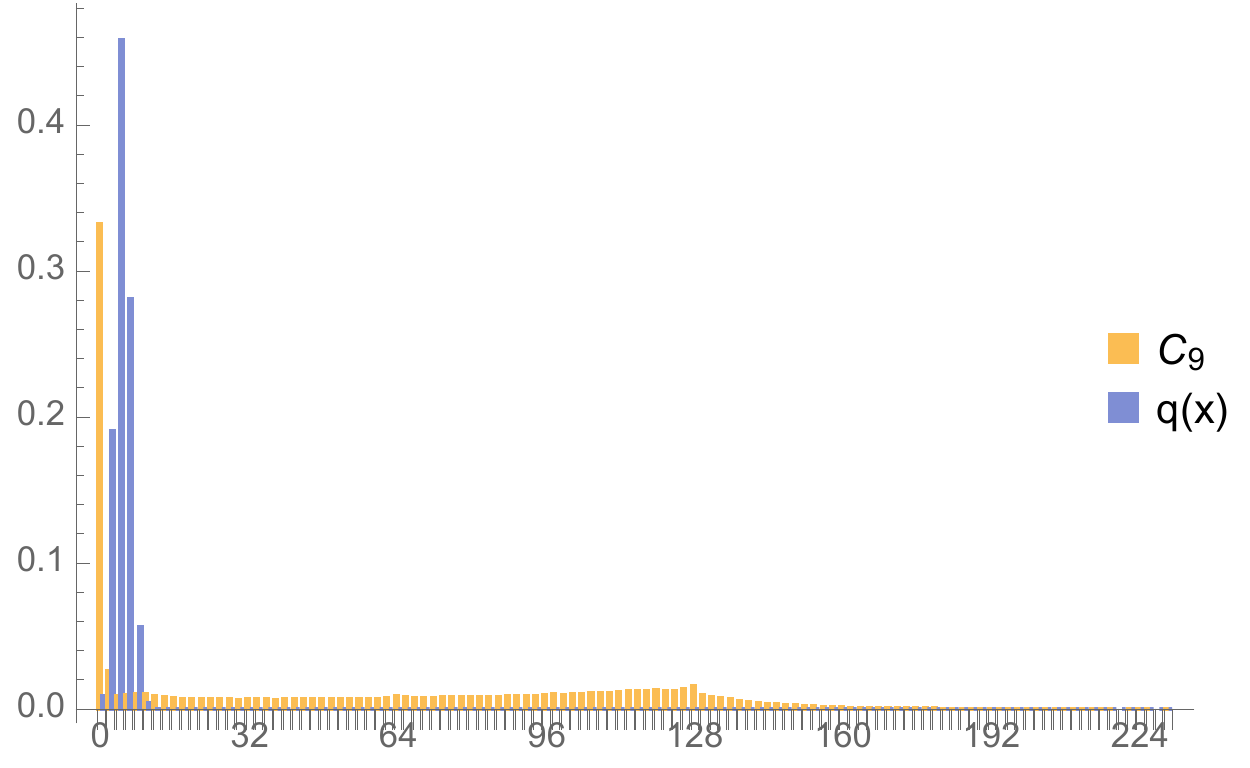}}
\subfigure{\includegraphics[width=0.2\textwidth]{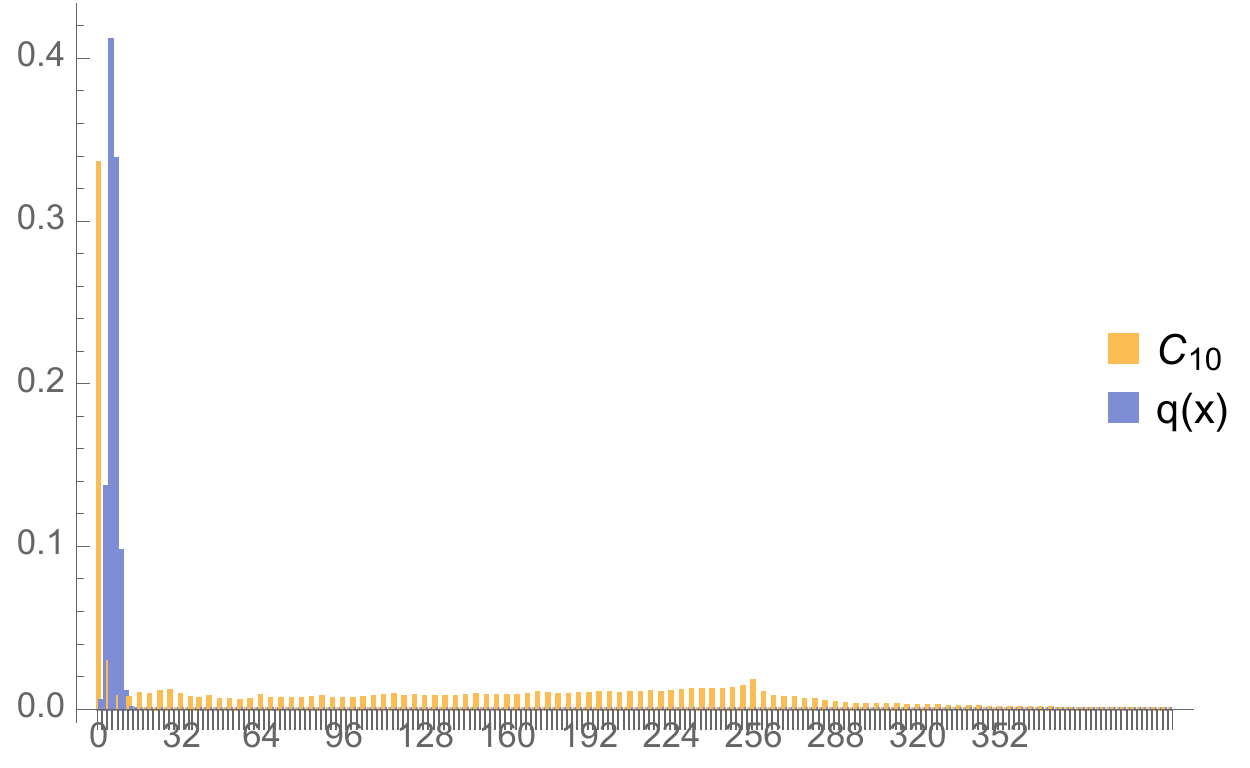}}
\caption{Distributions of the number of nontrivial real solutions to cyclic networks compared with the distribution of the number of real roots of a random polynomial of corresponding degree}
\label{cyclicnetworks}
\end{figure}

\subsubsection{Complete Networks}
We do a similar analysis here as in Section \ref{sec431}. Here we compare the number of nontrivial real solutions to power flow equations of complete $4-8$ node networks to that of polynomials of degree $12,54,220,860$ and $3304$, the generic number of nontrivial complex solutions for each network respectively. Again, we run numerical simulations using the $\mathtt{CountRoots}$ function in Mathematica to compute $10,000$ trials for $q$. We compare this against the distributions calculated in Section \ref{sec42}. Distribution results are given graphically in Figure \ref{completenetworks} and in more detail in the Appendix in Tables~\ref{k4dist},\ref{k5dist},\ref{k6dist},\ref{k7dist},\ref{k8dist}. Expected values are given in Table \ref{expectedrealcomplete}. We see in Figure \ref{completenetworks} that as the number of vertices grows, the distribution of the number of nontrivial real solutions to the power flow equations shifts much further right than for random polynomials. This is reflected in the expected values as the expected number of nontrivial real solutions for $K_5,K_6,K_7,K_8$ is much higher than that for a random polynomial. For $K_8$ we see that the expected number of nontrivial real solutions is over $28$ times as high as that for a random polynomial of degree $3304$.

While the expected number of nontrivial real solutions to the power flow equations is much higher than that of a random polynomial, it is easy to construct a univariate polynomial of degree $N$ which has $N$ real solutions by taking $q(x) = \Pi_{i=1}^N (x - \alpha_i)$ where $\alpha_i \in \mathbb{R}$ and $\alpha_i \neq \alpha_j$ for $i \neq j$. In contrast, recall that the maximal number of nontrivial real solutions to the power flow equations is an open question. 

\begin{figure}
\centering    
\subfigure{\includegraphics[width=0.2\textwidth]{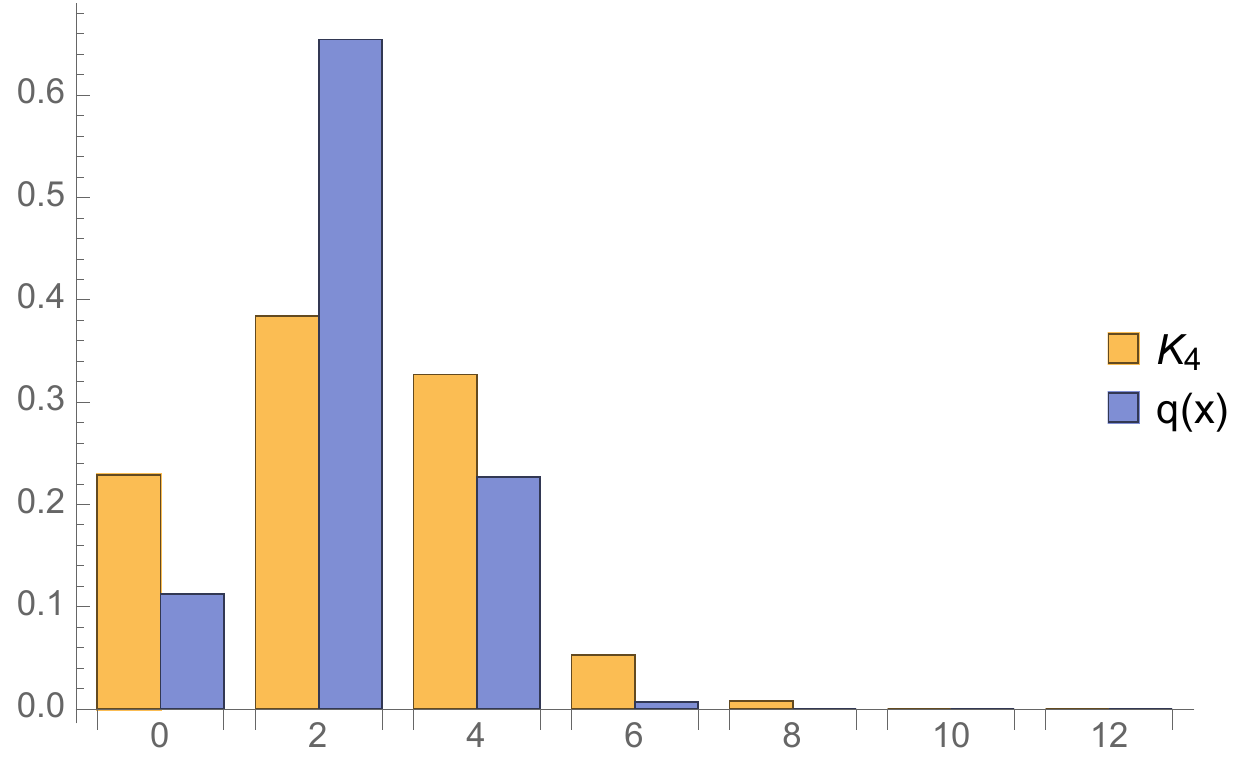}}
\subfigure{\includegraphics[width=0.2\textwidth]{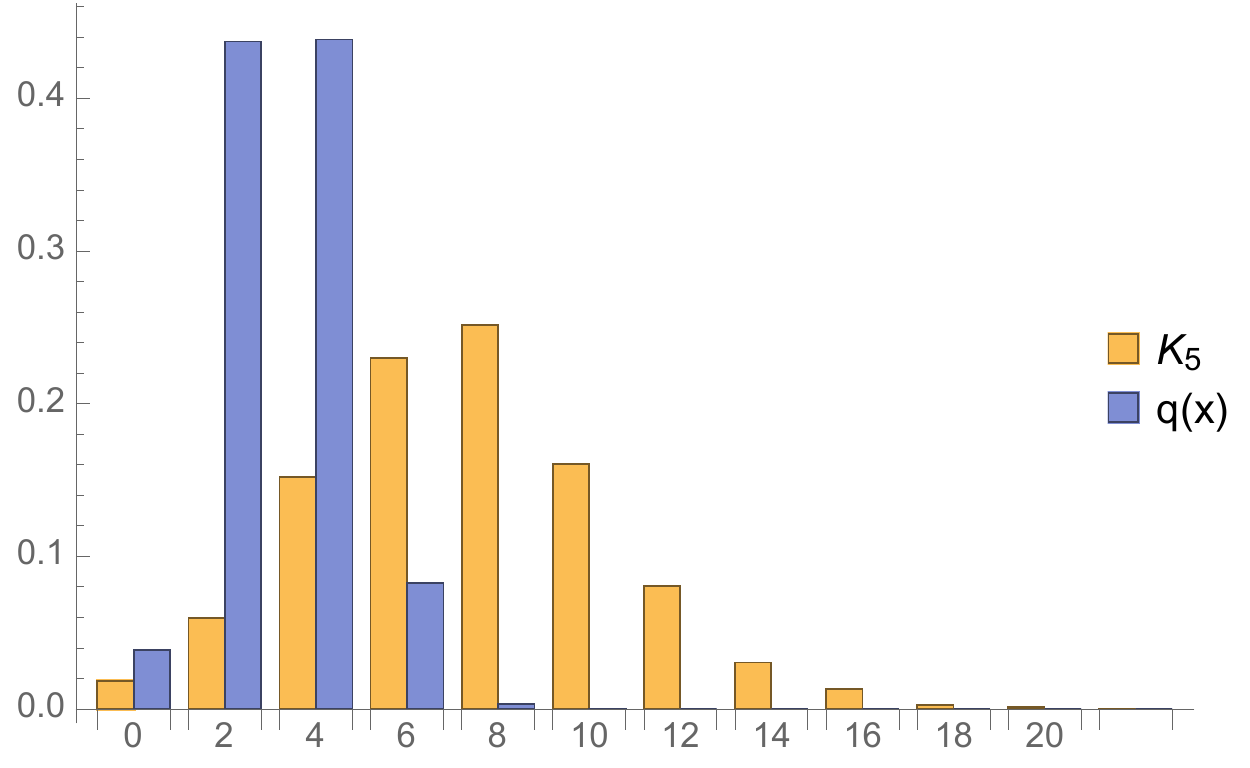}}
\subfigure{\includegraphics[width=0.2\textwidth]{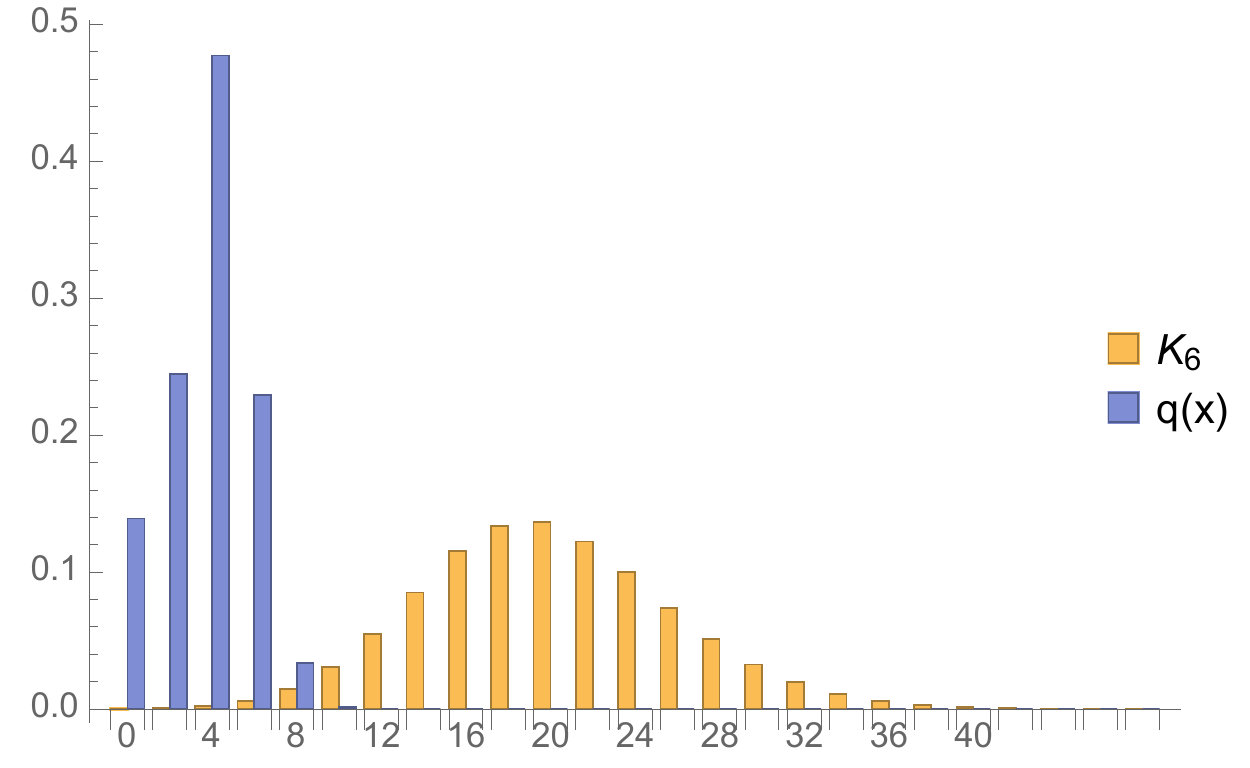}}
\subfigure{\includegraphics[width=0.2\textwidth]{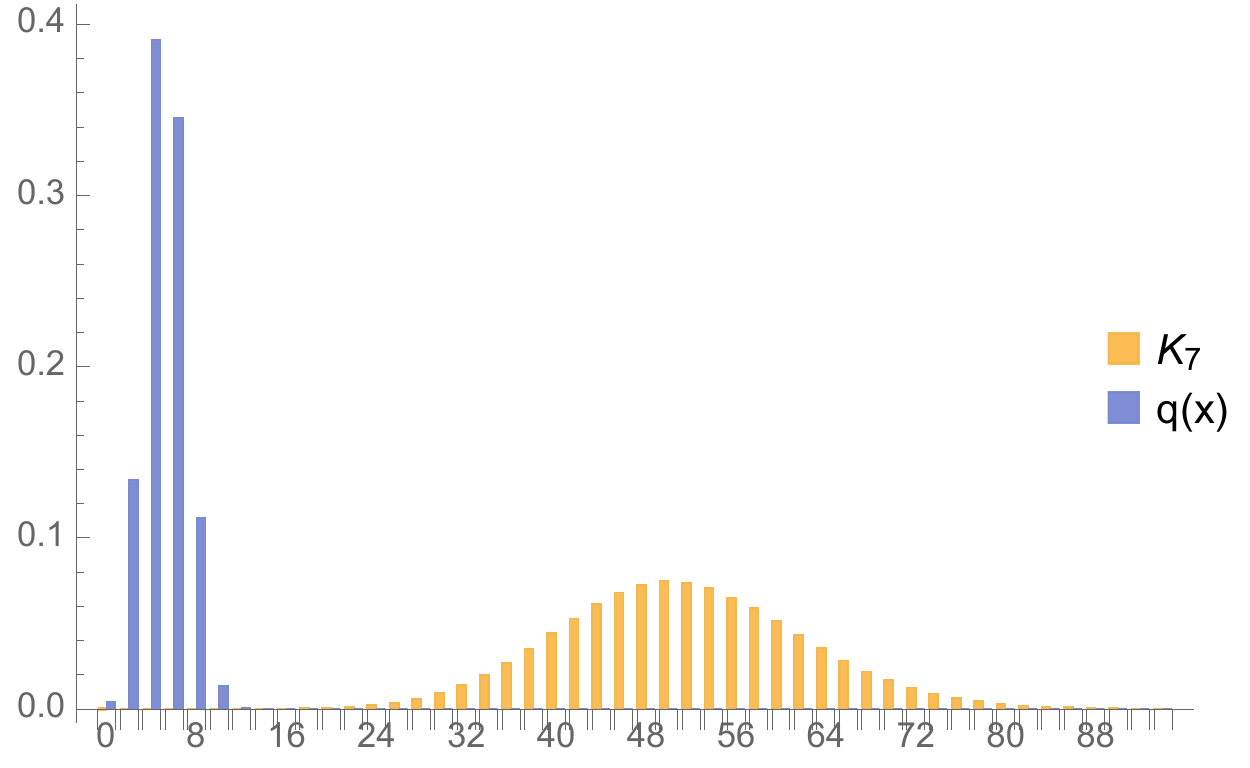}}
\subfigure{\includegraphics[width=0.2\textwidth]{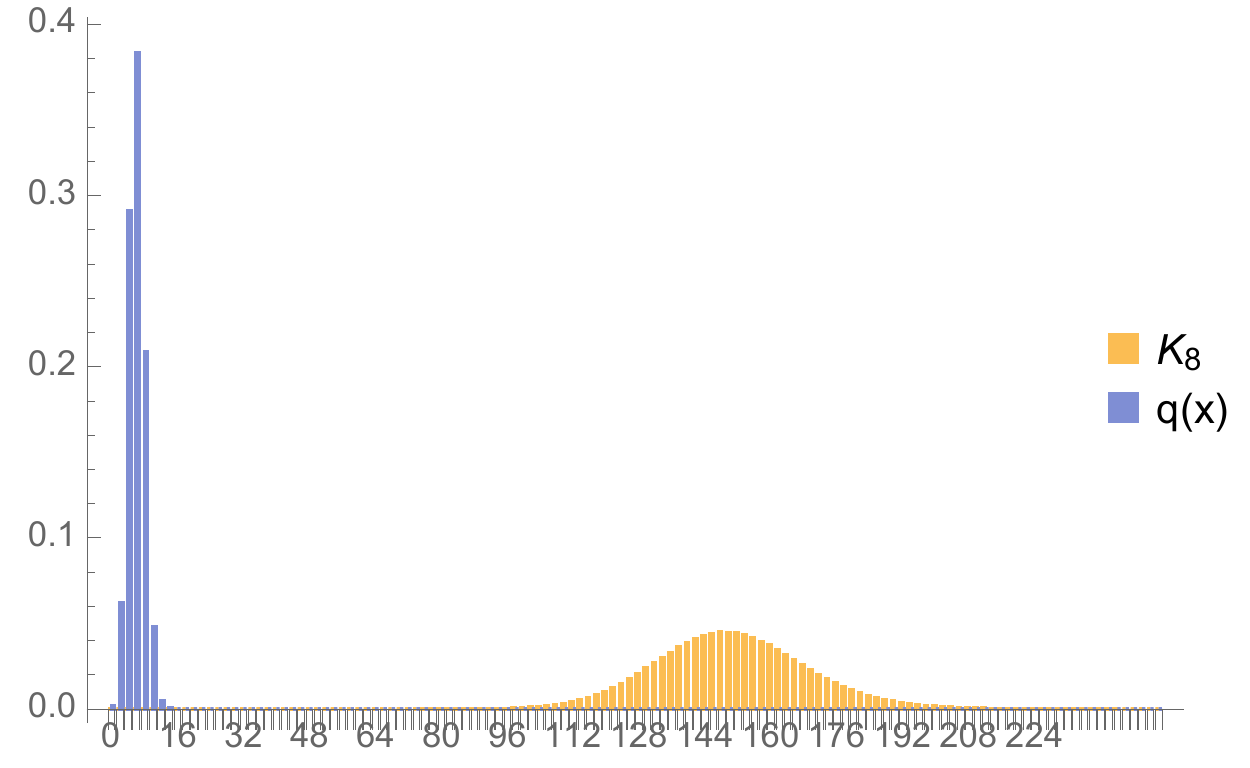}}
\caption{Distributions of the number of nontrivial real solutions to complete networks compared with the distribution of the number of real roots of a random polynomial of corresponding degree}
\label{completenetworks}
\end{figure}

\begin{table}
\caption{Expected Number of Nontrivial Real Solutions to Complete Networks}
\label{expectedrealcomplete}
\centering
\begin{tabular}{||c| c||c |c||c |c||}
\hline
 $K_3$ & $0.62$ & $K_4$ & $2.45$ & $K_5$ & $7.41$  \\
 $q(x)$ & $1.30$ & $q(x)$ & $2.26$ & $q(x)$ & $3.18$ \\
\hline
 $K_6$ & $20.11$ & $K_7$ &$51.54$ & $K_8$ & $150.65$ \\
 $q(x)$ & $4.06$ & $q(x)$ &$4.93$ & $q(x)$ & $5.30$ \\
 \hline 
\end{tabular}
\end{table}

\section{Other Families of Solutions} \label{sec5}
\subsection{Infinitely Many Solutions}
While it has been proven that for a generic choice of susceptance values, \eqref{pfeq1}-\eqref{pfeq2} admit finitely many solutions \cite{mallada2016distributed}, it is still interesting to note that there are simple examples where this fails. We extend the results of Lemma \ref{cmod4infinite}.
\begin{lem}
There are susceptance values for $K_n$ with $n\geq 4$ even that admit infinitely many real solutions.
\end{lem}{}
\begin{proof}
Set all susceptances $b_{km} = 1$. For $n$ even it can be verified that there exists a family of solutions of the form
\begin{align*}
    \{y_1 = 0, y_{k+1} = -y_k, x_1 = -1, x_{k+1} = -x_k \}
\end{align*}
for all even $k\geq 2$. Choosing $x_m,y_m \in (0,1)$ such that $x_m^2+y_m^2=1$ for all odd $m \geq 1$ gives infinitely many real solutions.

For $n$ odd it can be verified that there exists a family of solutions of the form
\begin{align*}
    \{& y_1 = 0, y_2 = \frac{\sqrt{3}}{2}x_3 - \frac{1}{2}y_3, y_4 = -\frac{\sqrt{3}}{2}x_3- \frac{1}{2}y_3,\\
    &y_k = -y_{k-1}, x_1 = -1, x_2 = \frac{1}{2}x_3 + \frac{\sqrt{3}}{2}y_3, \\
    &x_4 = \frac{1}{2}x_3 - \frac{\sqrt{3}}{2}y_3, x_k = -x_{k-1} \}
\end{align*}
for $k$ even and $k \geq 6$. Choosing $x_m, y_m \in (0,1)$ such that $x_m^2+y_m^2 = 1$ for all odd $m \geq 3$ gives infinitely many real solutions.
\end{proof}{}

\subsection{Only Trivial Solutions}

\begin{lem} \label{treelemma}
 Tree networks admit only trivial solutions.
\end{lem}{} 
\begin{proof}
Let $T = (V,E)$ be a tree  and suppose $T$ has $s$ vertices with degree equal to $1$ and $t$ vertices with degree greater than or equal to one. Since $T$ is a tree, $s \geq 1$. This gives equations
\begin{align}
    \sum_{m \text{ adjacent to } k} b_{km} \sin (\theta_k - \theta_m) &= 0 \ \text{ for all } v_k, \ \deg(v_k) \geq 1 \label{nonleafeqs}\\
    b_{km} \sin (\theta_k - \theta_m) &= 0 \ \text{ for all } v_k, \ \deg(v_k) = 1. \label{leafeqs}
\end{align}
Equation $(\ref{leafeqs})$ gives that at all vertices with degree $1$, $\sin(\theta_k - \theta_m) = 0$. If $|V| > 2$ for each vertex of degree $1$, we know it must be adjacent to at least one vertex of degree greater than $1$. This means we can rewrite $(\ref{nonleafeqs}) - (\ref{leafeqs})$ as
\begin{align}
    \sum_{\substack{m \text{ adjacent to } k \\ \deg(v_k) \geq 2}} \ b_{km} \sin(\theta_k - \theta_m) &= 0 \ \text{ for all } v_k, \ \deg(v_k) \geq 1 \label{nonleafeqs2}\\
    b_{km} \sin (\theta_k - \theta_m) &= 0 \quad \text{ for all } v_k, \  \deg(v_k) = 1. \label{leafeqs2}
\end{align}

The equations $(\ref{nonleafeqs2})-(\ref{leafeqs2})$ are now sparser than $(\ref{nonleafeqs})-(\ref{leafeqs})$. The equations defined in $(\ref{nonleafeqs2})$ are the same as those on tree $T = (V', E')$ where $V' = \{v \in V : \deg(v) >1 \}$ and $E' = \{e \in E : e \text{ is adjacent to } v_k, v_m \in V'\}$. Since $T$ was a tree and $T'$ is a subgraph of $T$, this means $T'$ is also a tree. We can repeat this argument again on $T'$ and so on until we are left with a system of equations where each equation only involves one term, $b_{km} \sin (\theta_k - \theta_m)$. The equation at $v_1$ simplifies to $b_{01} \sin(\theta_1) =0$ so $\theta_1 = n \pi$ for some $n \in \mathbb{Z}$. This forces $y_1 = \sin(\theta_1) = 0$ and $x_1 = \cos(\theta_1) = \pm 1$. We also have that for all $v_l$ adjacent to $v_1$ that $b_{1l} \sin(\theta_l - \theta_1) = 0$  so $\theta_1 - \theta_l = n \pi$ for some $n \in \mathbb{Z}$. This implies $\theta_l = n' \pi $ for $n' \in \mathbb{Z}$ giving that $y_l = 0$ and $x_l = \pm 1$. This argument repeats for all vertices adjacent to $v_l$ and so on. Since $T$ is connected, this covers all vertices $v \in V$.
\end{proof}

\begin{cor}
    Solution sets for tree networks are always zero dimensional.
\end{cor}{}

\begin{rem}
We note that Lemma \ref{treelemma} would follow from a result proven in \cite{chen2018counting} with the assumption that the variety is zero dimensional. In contrast, the proof provided here does not rely on this assumption.
\end{rem}{}

\section{Conclusion}
In this paper we presented a new method to calculate the distribution of number of real solutions of lossless power networks with all PV buses. We showed that this method is dramatically faster than standard homotopy methods. We compared the distribution of real solutions of the power flow equations to that of random polynomials and found that the power flow equations admitted many more real solutions. We also showed that for cyclic graphs the power flow equations can achieve the maximal bound of $n \binom{n-1}{\lfloor \frac{n-1}{2}\rfloor}$ real solutions and tree networks only achieve trivial real solutions. Finally, we gave explicit susceptance values for complete networks with more than three vertices and cyclic networks with $n \equiv 0 \mod 4$ vertices that give infinitely many real solutions.

\section{Acknowledgements}
The authors gratefully thank Jose Israel Rodriguez for his helpful comments and insight and gratefully  acknowledge  support  from  the  National  Science Foundation under grant DMS 1735928.

\bibliographystyle{unsrt}
\bibliography{IEEEtrans}

\section{Appendix: Distribution Data} \label{appendixdata}
\begin{table}[h!]
\caption{Distribution of Number of Nontrivial Real Solutions to $C_5$}
\label{c5dist}
\centering
\begin{tabular}{|c|c|c|c|c|}
\hline
$\#$ Real Solutions & 0 & 2 & 4 & 6  \\
Percentage Occurrence & 46.14 &13.62 &10.07& 14.84   \\
\hline
$\#$ Real Solutions & 8 & 10 & 12 & 14 \\
Percentage Occurrence &12.79 &2.45 &0.06& 0.03 \\
\hline 
\end{tabular}
\end{table}

\begin{table}[h!]
\caption{Distribution of Number of Nontrivial Real Solutions to $C_6$}
\label{c6dist}
\centering
\begin{tabular}{|c|c|c|c|c|}
\hline
$\#$ of Real Solutions & 0 & 4 & 8 & 12  \\
$\%$ of Occurrence & 44.03 & 15.84 &9.19 & 10.98  \\
\hline
$\#$ of Real Solutions& 16 & 20 & 24 & 28 \\
$\%$ of Occurrence & 14.13 & 4.67 & 0.03 & 0.13 \\
\hline 
\end{tabular}
\end{table}

\begin{table}[h!]
\caption{Distribution of Number of Nontrivial Real Solutions to $C_7$}
\label{c7dist}
\centering
\begin{tabular}{|c|c|c|c|c|c|}
\hline
$\#$  Real Solutions & 0 & 2 & 4 & 6 & 8   \\

$\%$  Occurrence & 42.61 & 4.48 & 2.01 & 2.26 & 2.66    \\
\hline
$\#$  Real Solutions & 10 & 12 & 14 & 16 & 18  \\
$\%$  Occurrence  & 2.73& 2.61 & 2.51 & 2.92& 2.96  \\
\hline 
$\#$  Real Solutions  &20 & 22 & 24 & 26  &28 \\
$\%$  Occurrence & 2.93 & 3.18& 3.79 & 4.11 & 4.10 \\
\hline 
$\#$  Real Solutions  & 30 & 32 & 34 & 36 & 38  \\
$\%$  Occurrence  & 4.33 & 4.44 & 2.45& 1.22 & 0.66   \\
\hline 
$\#$  Real Solutions & 40 & 42 & 44 & 46 & 48  \\
$\%$  Occurrence& 0.28 & 0.26 & 0.23 & 0.13 & 0.09   \\
\hline 
$\#$  Real Solutions& 50 & 52& $\geq 54$ &  &      \\
$\%$  Occurrence & 0.02 & 0.03 & 2e-4 &  &  \\
 \hline 
\end{tabular}
\end{table}

\begin{table}[h!]
\caption{Distribution of Number of Nontrivial Real Solutions to $C_8$}
\label{c8dist}
\centering
\begin{tabular}{|c|c|c|c|c|c|}
\hline
$\#$  Real Solutions & 0 & 4 & 8 & 12 & 16   \\
$\%$  Occurrence & 37.08 & 5.62 & 1.86 & 1.78 & 2.93 \\
\hline
$\#$  Real Solutions & 20 & 24 & 28 & 32 & 36 \\
$\%$  Occurrence & 2.95& 3.50 & 3.91 & 2.76 & 2.64 \\
\hline 
$\#$  Real Solutions & 40 & 44 & 48 & 52 & 56  \\
$\%$  Occurrence & 3.28 & 4.12 & 3.41 & 3.61 & 3.60 \\
\hline 
$\#$  Real Solutions & 60 & 64 & 68 & 72 & 76\\
$\%$  Occurrence  & 4.25 & 5.60 & 2.47 & 1.29 & 1.25 \\
\hline 
$\#$  Real Solutions &  80 & 84 & 88 & 92 & 96  \\
 $\%$  Occurrence&  0.85 & 0.54 & 0.44 & 0.06 &  0.09 \\
 \hline 
$\#$  Real Solutions & 100 & 104 & 108 & 112 & $\geq116$ \\
$\%$  Occurrence&  0.02 & 0.06 &  0.01 & 0.03 & 0.01 \\
 \hline 
\end{tabular}
\end{table}

\begin{table}[h!]
\caption{Distribution of Number of Nontrivial Real Solutions to $C_9$}
\label{c9dist}
\centering
\begin{tabular}{|c|c|c|c|c|c|}
\hline
$\#$  Real Solutions & 0 & 2 & 4 & 6 & 8   \\

$\%$  Occurrence & 33.24 & 2.64 & 0.93 & 0.96 &  1.00   \\
\hline
$\#$  Real Solutions & 10 & 12 & 14 & 16 & 18  \\
$\%$  Occurrence  & 1.00 &  0.91 & 0.81 & 0.77 & 0.73 \\
\hline 
$\#$  Real Solutions  &20 & 22 & 24 & 26  & 28  \\
$\%$  Occurrence & 0.68 & 0.68 & 0.71 & 0.71 & 0.68   \\
\hline 
$\#$  Real Solutions & 30 & 32 & 34 & 36 & 38  \\
$\%$  Occurrence  & 0.65 & 0.72 & 0.71 & 0.67 & 0.67 \\
\hline 
$\#$  Real Solutions & 40 & 42 & 44 & 46 & 48  \\
$\%$  Occurrence&   0.67 & 0.68 & 0.69 & 0.68 & 0.69 \\
\hline 
$\#$  Real Solutions& 50 & 52& 54  & 56 & 58      \\
$\%$  Occurrence &  0.69 & 0.70 & 0.72 & 0.72 & 0.72 \\
 \hline 
 $\#$  Real Solutions& 60 & 62& 64  & 66 & 68      \\
$\%$  Occurrence &  0.72 & 0.74 & 0.91 & 0.81 & 0.77 \\
 \hline 
 $\#$  Real Solutions& 70 & 72& 74  & 76 & 78      \\
$\%$  Occurrence & 0.75 & 0.80 & 0.82 & 0.81 & 0.82  \\
 \hline 
 $\#$  Real Solutions& 80 & 82& 84  & 86 & 88      \\
$\%$  Occurrence &  0.82 & 0.83 & 0.85 & 0.87 & 0.92 \\
 \hline 
 $\#$  Real Solutions& 90 & 92& 94  & 96 & 98      \\
$\%$  Occurrence & 0.91 & 0.90 & 0.93 & 1.00 & 1.01 \\
 \hline 
 $\#$  Real Solutions& 100 & 102& 104  & 106 & 108      \\
$\%$  Occurrence & 0.99 & 1.01 & 1.06 & 1.08 & 1.10  \\
 \hline 
 $\#$  Real Solutions& 110 & 112& 114  & 116 & 118      \\
$\%$  Occurrence &  1.14 &  1.17 &  1.22 &  1.24 &  1.27 \\
 \hline 
  $\#$  Real Solutions& 120 & 122& 124  & 126 & 128      \\
$\%$  Occurrence &  1.30 &  1.27 &  1.27 & 1.37 & 1.59  \\
 \hline 
   $\#$  Real Solutions& 130 & 132& 134  & 136 & 138      \\
$\%$  Occurrence & 0.94 & 0.82 & 0.75 & 0.67 & 0.60  \\
 \hline 
   $\#$  Real Solutions& 140 & 142& 144  & 146 & 148      \\
$\%$  Occurrence & 0.50 & 0.40 & 0.34 & 0.33 &  0.31 \\
 \hline 
   $\#$  Real Solutions& 150 & 152& 154  & 156 & 158      \\
$\%$  Occurrence & 0.30 & 0.25 & 0.22 & 0.18 &  0.16  \\
 \hline 
   $\#$  Real Solutions& 160 & 162& 164  & 166 & 168      \\
$\%$  Occurrence & 0.13 & 0.11 & 0.10 & 0.10 &  0.09  \\
 \hline 
   $\#$  Real Solutions& 170 & 172& 174  & 176 & 178      \\
$\%$  Occurrence &0.08 & 0.07 & 0.07 & 0.07 &  0.06   \\
 \hline 
   $\#$  Real Solutions& 180 & 182& 184  & 186 & 188      \\
$\%$  Occurrence & 0.05 & 0.05 & 0.04 & 0.04 & 0.03  \\
 \hline 
   $\#$  Real Solutions& 190 & 192& 194  & 196 & 198      \\
$\%$  Occurrence & 0.03 & 0.03 & 0.02 &  0.02  & 0.03  \\
 \hline 
    $\#$  Real Solutions& 200 & 202& 204  & 206 & $\geq 208$      \\
$\%$  Occurrence & 0.02 & 0.01 & 0.01 &  0.01  & 0.06 \\
 \hline 
\end{tabular}
\end{table}

\begin{table}[h!]
\caption{Distribution of Number of Nontrivial Real Solutions to $C_{10}$}
\label{c10dist}
\centering
\begin{tabular}{|c|c|c|c|c|c|}
\hline
$\#$  Real Solutions & 0 & 4 & 8 & 12 & 16   \\
$\%$  Occurrence & 33.60 & 2.91 & 0.76  & 0.67 & 0.93 \\
\hline
$\#$  Real Solutions & 20 & 24 & 28 & 32 & 36 \\
$\%$  Occurrence & 0.89 & 1.05 &1.12 & 0.90 & 0.70 \\
\hline 
$\#$  Real Solutions & 40 & 44 & 48 & 52 & 56  \\
$\%$  Occurrence & 0.64 & 0.74 & 0.58 & 0.58 & 0.53 \\
\hline 
$\#$  Real Solutions & 60 & 64 & 68 & 72 & 76\\
$\%$  Occurrence  & 0.58 &  0.80 & 0.65 & 0.64 & 0.64 \\
\hline 
$\#$  Real Solutions &  80 & 84 & 88 & 92 & 96  \\
 $\%$  Occurrence& 0.62 & 0.66 & 0.75 & 0.64 & 0.66 \\
 \hline 
$\#$  Real Solutions & 100 & 104 & 108 & 112 & 116 \\
$\%$  Occurrence&  0.66 & 0.71 & 0.74 & 0.80 & 0.87 \\
 \hline 
$\#$  Real Solutions & 120 & 124 & 128 & 132 & 136 \\
$\%$  Occurrence& 0.73 &  0.78 & 0.78 & 0.73 & 0.76  \\
 \hline 
$\#$  Real Solutions & 140 & 144 & 148 & 152& 156  \\
$\%$  Occurrence& 0.73 &  0.79 & 0.87 & 0.81 & 0.79 \\
 \hline 
 $\#$  Real Solutions & 160 & 164 & 168 & 172& 176  \\
$\%$  Occurrence& 0.81 &  0.83 & 0.86 & 1.0 & 0.94 \\
 \hline 
 $\#$  Real Solutions & 180 & 184 & 188 & 192& 196  \\
$\%$  Occurrence& 0.88 & 0.88 & 0.92 & 0.96 & 1.0 \\
 \hline 
 $\#$  Real Solutions & 200 & 204 & 208 & 212& 216  \\
$\%$  Occurrence& 0.99& 0.96 &  0.99 & 1.0 & 1.03 \\
 \hline 
 $\#$  Real Solutions & 220 & 224 & 228 & 232& 236  \\
$\%$  Occurrence& 1.02 & 1.04 & 1.14 & 1.19 & 1.19 \\
 \hline 
 $\#$  Real Solutions  & 240 & 244 & 248& 252 & 256  \\
$\%$  Occurrence& 1.20 & 1.20 & 1.22 &1.37 & 1.75 \\
 \hline 
 $\#$  Real Solutions  & 260 & 264 & 268& 272 & 276  \\
$\%$  Occurrence& 1.01 & 0.77 & 0.71 &  0.67 & 0.59 \\
 \hline 
 $\#$  Real Solutions  & 280 & 284 & 288& 292 & 296  \\
$\%$  Occurrence&  0.54 &  0.47 &  0.36 & 0.31& 0.29 \\
 \hline 
 $\#$  Real Solutions  & 300 & 304 & 308& 312 & 316  \\
$\%$  Occurrence& 0.26 &  0.25 &  0.25 & 0.23 & 0.22 \\
 \hline 
 $\#$  Real Solutions  & 320 & 324 & 328& 332 & 336  \\
$\%$  Occurrence& 0.21 & 0.19 &  0.18 & 0.15 & 0.14 \\
 \hline 
 $\#$  Real Solutions  & 340 & 344 & 348& 352 & 356  \\
$\%$  Occurrence& 0.12 &  0.11 & 0.09 &  0.10 &  0.09 \\
 \hline 
 $\#$  Real Solutions  & 360 & 364 & 368& 372 & 376  \\
$\%$  Occurrence& 0.07 &  0.07 &  0.06 &  0.05 &  0.04 \\
 \hline 
 $\#$  Real Solutions  & 380 & 384 & 388& 392 & 396 \\
$\%$  Occurrence&  0.04 &  0.03 & 0.03 & 0.02 &  0.02 \\
 \hline 
 $\#$  Real Solutions  & 400 & 404 & 408& 412 & $\geq 416$  \\
$\%$  Occurrence& 0.02 & 0.02 &  0.02 &  0.01 &  0.01 \\
 \hline 
\end{tabular}
\end{table}

\begin{table}[h!]
\caption{Distribution of Number of Nontrivial Real Solutions to $K_4$}
\label{k4dist}
\centering
\begin{tabular}{|c|c|c|c|c|c|c|c|c|}
\hline
$\#$ Real Solutions & 0 & 2 & 4 & 6 & 8 & 10 & 12 \\
$\%$ Occurrence & 22.91 & 38.38 &32.64& 5.32 & 0.74 &4e-6 & 0.00  \\
\hline
\end{tabular}
\end{table}

\begin{table}[h!]
\caption{Distribution of Number of Nontrivial Real Solutions to $K_5$}
\label{k5dist}
\centering
\begin{tabular}{|c|c|c|c|c|c|}
\hline
$\#$  Real Solutions & 0 & 2 & 4 & 6 & 8  \\
$\%$  Occurrence & 1.82 & 5.94 & 15.18 & 22.97 & 25.17  \\
\hline 
 $\#$ Real Solution & 10 & 12 & 14& 16 & 18   \\
$\%$ Occurrence & 16.03 & 8.04 & 3.05& 1.29 & 0.35  \\
\hline
 $\#$ Real Solution & 20 & 22 &  $\geq 24$ &  &   \\
 $\%$ Occurrence & 0.14 & 0.02 &3e-3 &  & \\
 \hline
\end{tabular}
\end{table}

\begin{table}[h!]
\caption{Distribution of Number of Nontrivial Real Solutions to $K_6$}
\label{k6dist}
\centering
\begin{tabular}{|c|c|c|c|c|c|}
\hline
$\#$  Real Solutions & 0 & 2 & 4 & 6 & 8   \\

$\%$  Occurrence & 0.02 & 0.06 & 0.22 & 0.59 & 1.48 \\
\hline
$\#$  Real Solutions & 10 & 12 & 14 & 16& 18   \\
$\%$  Occurrence  & 3.05 & 5.50 & 8.50 & 11.53& 13.35  \\
\hline 
$\#$ Real Solutions & 20 & 22 & 24 &26 & 28   \\
$\%$ Occurrence & 13.67 & 12.22 & 10.00 & 7.37 & 5.12  \\
\hline 
$\#$ Real Solutions & 30 & 32 & 34& 36 & 38 \\
$\%$ Occurrence& 3.24 & 1.95 & 1.07& 0.57 & 0.28  \\
\hline 
$\#$ Real Solutions & 40 & 42 & 44& 46 & $\geq 48$ \\
$\%$ Occurrence& 0.14 & 0.06 & 0.03 & 0.01 & 8e-5  \\
\hline 
\end{tabular}
\end{table}

\begin{table}[h!]
\caption{Distribution of Number of Nontrivial Real Solutions to $K_7$}
\label{k7dist}
\centering
\begin{tabular}{|c|c|c|c|c|c|}
\hline
$\#$  Real Solutions & 0 & 2 & 4 & 6 & 8   \\
$\%$  Occurrence & 0.00 & 0.00 & 2e-6 & 1e-6 & 9e-6   \\
\hline
$\#$  Real Solutions & 10 & 12 & 14 & 16 & 18  \\
$\%$  Occurrence & 1e-5 & 4e-5 & 7e-5 & 0.02 & 0.03    \\
\hline 
$\#$  Real Solutions  &20 & 22 & 24 & 26 & 28    \\
$\%$  Occurrence & 0.06& 0.13 & 0.22 & 0.37 & 0.60     \\
\hline 
$\#$  Real Solutions& 30 & 32 & 34 & 36 & 38  \\
$\%$  Occurrence & 0.93 & 1.39 & 1.97& 2.70 & 3.53\\
\hline 
$\#$  Real Solutions & 40& 42 & 44 & 46 & 48     \\
$\%$  Occurrence & 4.43 & 5.29& 6.14 & 6.81 & 7.25 \\
\hline 
$\#$  Real Solutions & 50 & 52  & 54& 56 & 58    \\
$\%$  Occurrence & 7.46 & 7.37& 7.05 & 6.49 & 5.89 \\
\hline 
$\#$ Real Solutions & 60 & 62& 64 & 66 &  68  \\
$\%$ Occurrence & 5.13 & 4.30 & 3.54& 2.84 & 2.20   \\
\hline 
$\#$ Real Solutions & 70   & 72 & 74& 76 & 78  \\
$\%$ Occurrence & 1.68    & 1.25 & 0.91 & 0.66 & 0.46\\
\hline
$\#$ Real Solutions& 80 & 82 & 84 &86& 88 \\
$\%$ Occurrence & 0.31 & 0.21 & 0.14 & 0.10&  0.06    \\
\hline
$\#$ Real Solutions& 90 & 92 & 94 &$\geq 96$ &   \\
$\%$ Occurrence& 0.004 & 0.03 & 0.02 & 3e-2 &  \\
\hline
\end{tabular}
\end{table}

\begin{table}[h!]
\caption{Distribution of Number of Nontrivial Real Solutions to $K_8$}
\label{k8dist}
\centering
\begin{tabular}{|c|c|c|c|c|c|}
\hline
$\#$  Real Solutions & $\leq 88$ & 90 & 92 & 94 & 96   \\
$\%$  Occurrence & 0.04 & 0.02 & 0.02& 0.03 & 0.04 \\
\hline
$\#$  Real Solutions & 98 & 100 & 102 & 104 & 106  \\
$\%$  Occurrence & 0.05& 0.08 & 0.10 & 0.14 &0.19   \\
\hline 
$\#$  Real Solutions &108 &110 & 112 & 114 & 116    \\
$\%$  Occurrence & 0.25 &0.32 & 0.42 & 0.54 &0.67   \\
\hline 
$\#$  Real Solutions& 118 & 120 & 122 & 124 & 126   \\
$\%$  Occurrence & 0.84 & 1.03  & 1.26 &   1.50 & 1.78  \\
\hline 
$\#$  Real Solutions & 128 & 130& 132 & 134 & 136     \\
$\%$  Occurrence & 2.06  & 2.39 & 2.69 & 2.99 & 3.29\\
\hline 
$\#$  Real Solutions & 138 & 140  & 142& 144 & 146    \\
$\%$  Occurrence &  3.61 & 3.85 & 4.13 & 4.26 & 4.41  \\
\hline 
$\#$ Real Solutions & 148 & 150& 152 & 154 &  156  \\
$\%$ Occurrence & 4.50 & 4.46 & 4.43  & 4.33 & 4.15   \\
\hline 
$\#$ Real Solutions & 158   & 160 & 162& 164 & 166  \\
$\%$ Occurrence & 3.95 & 3.74 & 3.44 & 3.17 & 2.90  \\
\hline
$\#$ Real Solutions& 168 & 170 & 172 &174& 176 \\
$\%$ Occurrence & 2.59 & 2.30 & 2.014 & 1.77 & 1.53   \\
\hline
$\#$ Real Solutions& 178 & 180 & 182 &184 &186   \\
$\%$ Occurrence& 1.32 & 1.12 & 0.95 & 0.79 & 0.66 \\
\hline
$\#$ Real Solutions& 188 & 190 & 192 &194 & 196  \\
$\%$ Occurrence&  0.55 & 0.46 & 0.37 & 0.31 & 0.24 \\
\hline
$\#$ Real Solutions& 198 & 200 & 202 &204 &206   \\
$\%$ Occurrence& 0.21 & 0.16 & 0.13 & 0.11 & 0.08 \\
\hline
$\#$ Real Solutions& 208 & 210 & 212 &214 & 216  \\
$\%$ Occurrence& 0.07 & 0.05 & 0.04 & 0.03 & 0.03  \\
\hline
$\#$ Real Solutions& 218 & 220 & 222 &224 &  $\geq 226$ \\
$\%$ Occurrence&  0.02 & 0.02 & 0.01 & 0.01 &0.03 \\
\hline
\end{tabular}
\end{table}

\end{document}